\documentclass[10pt]{amsart}

\usepackage[latin1]{inputenc}  
\usepackage{lmodern}  
\usepackage{amsmath}
\usepackage{amssymb}
\usepackage{hyperref}
\usepackage{commath}
\usepackage{xcolor}

\DeclareMathAlphabet{\mathpzc}{OT1}{pzc}{m}{it}
\usepackage[all]{xy}
\usepackage{latexsym}
\usepackage{graphicx}
\usepackage{tikz}
\usepackage{transparent}
\usepackage{graphics}

\usepackage[left=2.7cm,right=2.7cm,top=2.6cm,bottom=3cm]{geometry}

\newtheorem{theorem}{Theorem}[section]         
\newtheorem{lemma}[theorem]{Lemma}

\newtheorem{prop}[theorem]{Proposition}

\theoremstyle{remark}
\newtheorem{remark}[theorem]{Remark}

\newtheorem{notation}[theorem]{Notation}

\theoremstyle{definition}
\newtheorem{defi}[theorem]{Definition} 

\newcommand{\T}{\bf{T}}

\newcommand{\calA}{\mathcal{A}}
\newcommand{\calB}{\mathcal{B}}
\newcommand{\calC}{\mathcal{C}}
\newcommand{\calR}{\mathcal{R}}
\newcommand{\calX}{\mathcal{X}}
\newcommand{\calY}{\mathcal{Y}}
\newcommand{\calZ}{\mathcal{Z}}
\newcommand{\calF}{\mathcal{F}}
\newcommand{\frakF}{\mathfrak{F}}
\newcommand{\frakG}{\mathfrak{G}}

\newcommand{\Sur}{\bf{S}}
\newcommand{\tT}{\tilde{\T}}

\usepackage{color}

\title{Mutation of type $D$ friezes}
\author{A. Garcia Elsener and K. Serhiyenko}

\AtEndDocument{\bigskip{\footnotesize%

  \textsc{Institute for Mathematics and Scientific Computing - NAWI Graz, University of Graz, Heinrichstra{\ss}e 36, A-8010, Austria.} \par
  \textit{E-mail address} \texttt{ana.garcia-elsener@uni-graz.at} \par

  \textsc{Department of mathematics, University of Kentucky, Lexington, KY 40506, USA.} \par
  \textit{E-mail address} \texttt{khrystyna.serhiyenko@uky.edu}

}}

\begin{document}

\maketitle

\begin{abstract}
In this article we study mutation of friezes of type $D$.  We provide a combinatorial formula for the entries in a frieze after mutation.  The two main ingredients in the proof include a certain transformation of a type $D$ frieze into a sub-pattern of a frieze of type $A$ and the mutation formula for type $A$ friezes recently found by Baur et al. 
\end{abstract}

\section{Introduction}
Friezes of type $A$ are beautiful combinatorial objects that were first introduced and studied in 1970's by Conway and Coxeter \cite{C1,CC1,CC2}.  Interest in friezes renewed after the development of cluster algebras in 2001, as the two notions proved to be related \cite{Assem1, CCh}.  Motivated by this newfound connection, friezes of type $D$, $E$, $\widetilde{A}$ and various other generalizations appeared in the literature, see for example \cite{M-G} and references therein.

For every quiver $Q$, without oriented cycles of length two,  there is an associated cluster algebra $\mathcal{A}_Q$.  The algebra $\mathcal{A}_Q$ is contained in the field of rational functions $\mathbb{Q}(x_1, x_2, \dots, x_n)$, where $n$ denotes the number of vertices in $Q$.  It is defined by a set of generators, called cluster variables, that are computed recursively via an operation called \emph{mutation} \cite{FZ}. When $Q$ is acyclic, the cluster category $\mathcal{C}_Q$ provides a categorification of $\mathcal{A}_Q$  \cite{BMRRT,CCS06}.  
This category admits a cluster-tilting object $T$, a rigid object of $\mathcal{C}_Q$ consisting of $n$ indecomposable nonisomorphic direct summands. 
Also, there is an associated cluster character $\chi_T: \text{obj}\,\mathcal{C}_Q\to \mathcal{A}_Q$ \cite{CCh} that maps rigid indecomposable objects in $\mathcal{C}_Q$ to cluster variables in $\mathcal{A}_Q$. Precomposing $\chi_T$ with the evaluation homomorphism that sets $x_i=1$ for all $i\in[1,n]$ we obtain a map $F: \text{obj}\,\mathcal{C}_Q\to \mathbb{Q}$.   Furthermore, the Laurent phenomenon and positivity proved in \cite{FZ} and \cite{LS} respectively, imply that $F$ maps objects of $\mathcal{C}_Q$ to positive integers.   In particular, if we take the Auslander--Reiten quiver of $\mathcal{C}_Q$, where the quiver $Q$ arises from a triangulation of a polygon or a once-punctured disk, and apply $F$ to every indecomposable we obtain a frieze of type $A$ or $D$, see \cite[Section 5]{CCh} and \cite{BM} respectively. This describes the relation between friezes and cluster algebras.

The mutation operation is crucial to the definition of cluster algebras.  Compatible notions of mutation also appear in the setting of surface triangulations, quivers, cluster-tilting objects, and cluster-tilted algebras.  Recently, the concept of frieze mutations was introduced and studied in \cite{baur2018mutation} for friezes of type $A$.  The  main formula is purely combinatorial, however its proof relies strongly on the particular properties of the associated cluster categories and the representation theory of cluster-tilted algebras.

Motivated by these results, in this article we study mutations of type $D$ friezes.  However, the combinatorics of surface triangulations and the associated representation theory is more complicated in type $D$.   In particular, the specialized cluster character map $F$ that sends objects of the cluster category $M$ to positive integers in type $A$ amounts to counting the number of submodules up to isomorphisms.  In type $D$, this is no longer the case, and one needs to compute Euler characteristic of the Grassmannians of submodules of $M$ of a given dimension vector.  Also, the formula in type $A$ relies on the fact that it is easy to compute the support of $\text{Hom}_{\mathcal{C}_Q}(M, -)$ 
and middle terms of extensions in the Auslander--Reiten quiver.  However, this becomes more complex in type $D$ and involves analyzing different cases. 

Therefore to avoid these difficulties, we first transform our initial type $D$ frieze into a sub-pattern of a larger one of type $A$. 

This approach also appears in \cite{Ma} in the study of the corresponding cluster variables. 
Next, in Theorem~\ref{mut-thm} we prove that this operation is compatible with mutations, where a single mutation in type $D$ yields two mutations in type $A$.   Applying the main result of \cite{baur2018mutation} we obtain the desired mutation formula for type $D$ friezes in Theorem \ref{main-theorem}.  The formula is presented in purely combinatorial terms and relies on the geometric models \cite{CCS06, BM, FST,S} and particular symmetries of this larger type $A$ frieze.

The article is organized as follows.  In Section \ref{Background} we recall the relevant background on friezes of type $A$ and $D$, and their mutation in type $A$.  Then in Section \ref{from type D}, we explain how to pass from type $D$ to type $A$ and prove that this operation behaves well with mutations.  Finally, we define the elements appearing in our main result and obtain  the mutation formula for type $D$ friezes in Section~\ref{mut-D}.  Here we also provide a detailed example.


\section{Background} \label{Background}
We begin by reviewing the relevant background information and establishing notation.  Here we recall the definition of friezes of types $A$ and $D$ and their construction from triangulations of polygons and once punctured disks respectively.  We also describe relations between friezes and arcs in a surface, and state the main theorem of  \cite{baur2018mutation} on mutation of friezes in type $A$. 

\subsection{Friezes of types $A$ and $D$} In this section we define friezes and review some key properties. 

\subsubsection{Friezes of type $A$}
\begin{defi}\label{def:A}
A \emph{frieze of type $A_{n}$} is a grid of integers $\mathfrak{F}$ as in Figure~\ref{friezeA}, with $n+4$ infinite rows, where the first and the last rows consist entirely of 0's, the second and the second to last rows consist entirely of 1's, and the remaining entries are positive integers.  In addition, the entries in $\mathfrak{F}$ satisfy the \emph{diamond rule} which says that for every diamond formed by the neighboring entries $\begin{smallmatrix}&a\\b&&c\\&d\end{smallmatrix}$ we have $bc-ad=1$.
\end{defi}

We say that the top and bottom rows of 0's and 1's are {\it trivial}.  Hence, a frieze of type $A_n$ consists of $n$ nontrivial rows.   It is shown in \cite[Problem 21]{CC2} that friezes of type $A_n$ have a horizontal period of $n+3$.    An example of a frieze of type $A_3$, which corresponds to the triangulation of hexagon given in Figure \ref{quiddity-sec}, is given in Figure~\ref{typeA-example}.

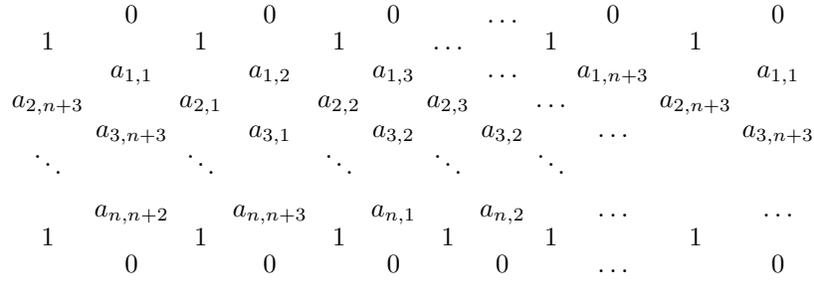
\begin{figure}
$$\xymatrix@C=-2pt@R=-2pt{&0&&0&&0&& \dots&&0&& 0\\
1&&1&&1&& \dots&&1&& 1\\
&a_{1,1}&& a_{1,2} && a_{1,3} && \dots && a_{1,n+3} && a_{1,1}\\
a_{2,n+3} && a_{2,1} && a_{2,2}&&a_{2,3} && \dots&& a_{2,n+3}\\
&a_{3,n+3} && a_{3,1} && a_{3,2}&&a_{3,2} && \dots&& a_{3,n+3}\\
 \ddots && \ddots &&  \ddots  && \ddots  && \ddots\\
&a_{n,n+2}& &a_{n,n+3} && a_{n,1} && a_{n,2} && \dots && \dots \\
1&& 1 && 1&&1 && 1&&1\\
&0&&0&&0&& 0 && \dots&& 0}$$
\caption{Frieze of type $A_n$.}
\label{friezeA}
\end{figure}

\begin{figure}[h!]
\centering
\scalebox{.9}{
\def\svgwidth{3.5in}
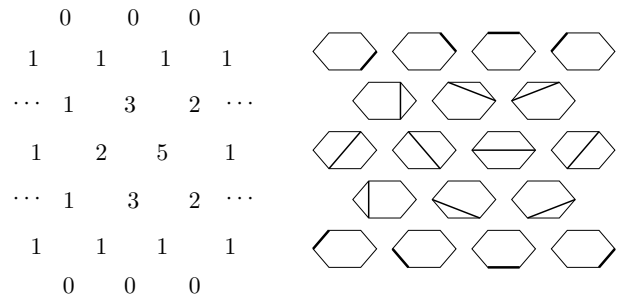}
\caption{Example of a frieze of type $A_3$ (left). Category of arcs (right).}
\label{typeA-example}
\end{figure}

\subsubsection{Friezes of type $D$}
\begin{defi}\label{def:D}
A \emph{frieze of type $D_{n}$} is a grid of integers $\mathfrak{F}$ as in Figure~\ref{friezeD}, with $n+2$ infinite rows, where the first row consist entirely of 0's, the second row consist entirely of 1's, and the remaining rows consist of positive integers.  The entries in the first $n$ rows of $\mathfrak{F}$ satisfy the {diamond rule}.   In addition, the entires in the last three rows satisfy the following relations. Whenever there is a configuration $\begin{smallmatrix}&a\\b&&c\\d&&e\end{smallmatrix}$ of neighboring entries then $bc-a=de-a=1$, and whenever there is $\begin{smallmatrix}&a\\b&&c\\&d\\&e\end{smallmatrix}$ then $bc-ade=1$.
\end{defi}

If $n$ is even then the frieze $\mathfrak{F}$ has horizontal period $n$.  If $n$ is odd then the first $n$ rows are $n$-periodic, but there is a twist happening in the last two rows, where the entries $a_{n-1,n}$ and $a_{n,n}$ appearing on the far right of the diagram in Figure~\ref{friezeD} switch places.    In this case $\mathfrak{F}$ is $2n$-periodic.  This follows from \cite[Theorem 1.1]{BM}. 
Similarly, we say that the first two rows of a type $D_n$ frieze are {\it trivial}.   Hence a $D_n$ frieze has $n$ nontrivial rows.    An example of a frieze of type $D_4$, which corresponds to the triangulation of the punctured disc given in Figure \ref{quiddity-sec}, is given in Figure~\ref{typeD-example}.

\begin{figure}[h!]
\centering
\scalebox{.95}{
\def\svgwidth{4.2in}
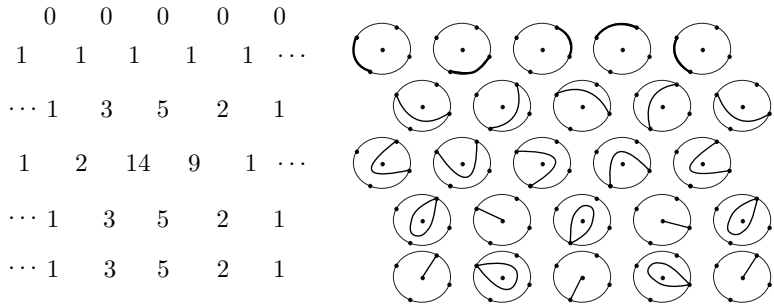}
\caption{Example of a frieze of type $D_4$ (left). Category of arcs (right).}
\label{typeD-example}
\end{figure}

We will sometimes refer to friezes of type $A_n$ and $D_n$ simply as friezes of type $A$ and $D$ respectively.  Moreover, we always consider friezes up to horizontal translation, and in type $D$ we also consider friezes up to interchanging the last two nontrivial rows.

\begin{figure}
$$\xymatrix@C=-3pt@R=-3pt{&0&&0&&0& \dots &0 && 0\\
1&&1&&1&& \dots&&1\\
&a_{1,1}&& a_{1,2}  && \dots && a_{1,n} && a_{1,1}\\
a_{2,n} && a_{2,1} && a_{2,2} && \dots&& a_{2,n}\\
& \ddots && \ddots &&  \ddots    && \ddots && \ddots\\
 &&a_{n-2,n} && a_{n-2,1} && a_{n-2,2} && \dots &&  \\
 &&&a_{n-1,n} && a_{n-1,1} && a_{n-1,2} && \dots && a_{n-1,n} \\
 &&&a_{n,n} && a_{n,1} && a_{n,2} && \dots && a_{n,n} \\
}$$
\caption{Frieze of type $D_n$ when $n$ is even.}
\label{friezeD}
\end{figure}
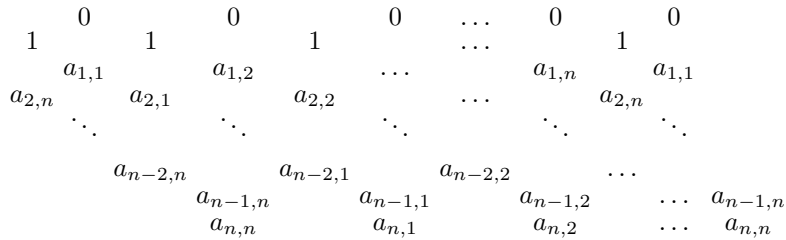


\subsection{Friezes from triangulations}

In general it is difficult to find an array of integers as described above that satisfies all the conditions. However, the friezes we study are related to triangulations of surfaces, which allows for their explicit construction. In particular, friezes of type $A$ are in bijection with triangulations of polygons, and friezes of type $D$ can be obtained from triangulations of a once punctured disk.  We review the necessary definitions below, following \cite{FST}. 

Let $\bf{S}$ denote an oriented Riemann surface, possibly with boundary $\partial \bf{S}$, and let $\bf{M}\subset\bf{S}$ be a finite subset of $\bf{S}$.  Furthermore, we require that each boundary component has nontrivial intersection with $\bf{M}$.   We call the elements of $\bf{M}$ \emph{marked points}, we call the elements of $\bf{M}$ that lie in the interior of $\bf{S}$ \emph{punctures}, and we call the pair $(\bf{S},\bf{M})$ a \emph{marked surface}.  In the subsequent sections we will only consider the case when the marked surface is a disk with $\abs{\bf{M}}\geq 4$ and at most one puncture.  However, we describe the general situation below to avoid cases. 

We define an \emph{arc} in $(\bf{S},\bf{M})$ to be a curve $\gamma$ in $\bf{S}$ such that its endpoints are marked points, and it is disjoint from $\bf{M}$ and $\partial \bf{S}$ otherwise.  Also, we require that $\gamma$ does not intersect itself, except possibly at the endpoints, and it is not contractable into $\bf{M}$ or into the boundary of $\bf{S}$.  An arc $\gamma$ is considered up to isotopy relative to the endpoints of $\gamma$.   We say that two arcs $\gamma_1$ and $\gamma_2$ in $\bf{S}$ are \emph{noncrossing} if there exist curves in their relative isotopy classes that are nonintersecting except possibly at their endpoints.  An 
\emph{ideal triangulation} $\bf{T}$ of $(\bf{S},\bf{M})$ is defined as a maximal collection of pairwise noncrossing arcs.

Given an ideal triangulation $\T$ of a disk (or punctured disk) $\Sur$ with $m$ marked points on the boundary, we define a sequence of positive integers as follows.
Let $R= \Sur \setminus \{ \gamma \colon \gamma \in \T \}$ be the collection of open regions inside $\Sur$. For $i \in [1,m]$ we define $a_i$ to be the number of regions in $R$ locally visible in a small neighborhood around $i$. See Figure \ref{quiddity-sec}.

\begin{figure}[h!]
\centering
\scalebox{.85}{
\def\svgwidth{3.6in}
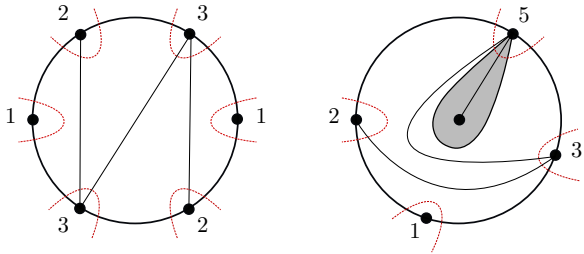}
\caption{Quiddity sequences: $(1,3,2,1,3,2)$ (left) and  $ (1,3,5,2)$ (right). These examples correspond to the friezes in Figures \ref{typeA-example} and \ref{typeD-example} respectively.}
\label{quiddity-sec}
\end{figure}

Thus, we obtain a sequence of $m$ positive integers moving counter-clockwise
around the boundary of the disk.  This sequence is called a \emph{quiddity sequence}, and it yields a first nontrivial row in a frieze $\frakF_{\T}$ with period $m$.  In the case of a disk without punctures, we can apply the diamond rule to obtain the entries in the second nontrivial row, then in the third row, and so on.  Continuing in this way we eventually obtain rows of 1's and 0's, the last two rows in $\frakF_{\T}$ by \cite[Problem 28]{CC2}.  This construction results in a frieze of type $A_{m-3}$ with $m-3$ nontrivial rows and period $m$.   Moreover, this gives a bijection between friezes of type $A_{m-3}$ and triangulations of $m$-gons \cite[Problems 28 and 29]{CC1,CC2}.  The frieze of type $A_3$ coming from a triangulation of a disk in Figure~\ref{quiddity-sec} is given in Figure~\ref{typeA-example}.

Similarly, in the case of a disk with one puncture, the quiddity sequence also yields the first row of the frieze of type $D_m$ with period $m$.  Again we obtain the first $m-2$ nontrivial rows via the diamond rule.  The remaining two rows can be uniquely determined by the other relations that define type $D_m$ frieze and the number of triangles in $\bf{T}$ incident to the puncture \cite[Proposition 6.8]{BM}.  The frieze of type $D_4$ coming from a triangulation of a punctured disk in Figure~\ref{quiddity-sec} is given in Figure~\ref{typeD-example}.

\begin{remark} \hfill
\begin{itemize}
\item[(a)]  In what follows, we will only work with friezes of type $D$ defined by quiddity sequences as above.
\item[(b)] Unlike in type $A$, there are friezes of type $D$ that do not come from quiddity sequences. See for example the frieze in \cite[Appendix A]{BM}.  
However, by \cite[Proposition A.2]{BM} any type $D$ frieze pattern with at least one 1 in the last two nontrivial rows arises from a quiddity sequence.
\end{itemize}
\end{remark}

\begin{remark}
Here we only consider ideal triangulations, also known as triangulations without tags.  In type $D$ it is possible to define a frieze from a tagged triangulation, however up to interchanging the last two nontrivial rows, it coincides with a frieze arising from an ideal triangulation.     
\end{remark}

From now on the term triangulation will always mean an ideal triangulation.

Above, we described how to obtain friezes from triangulations via quiddity sequences.  While the first nontrivial row of a frieze can be read off from the triangulation directly, the remaining entries in the frieze are computed recursively.  There is a more direct approach to obtain the entries in a frieze $\mathfrak{F}$, where each entry $a_{i,j}$ in $\mathfrak{F}$ comes from a corresponding arc $\gamma_{i,j}$ in the surface which we now describe. For the next part of the section we follow \cite{CCS06,S,MSW}.


\subsubsection{Category of arcs: associating a position in the frieze to each arc.}\label{arc_category}

In type $A_n$, this correspondence can be described as follows, where we use notation of Figure~\ref{friezeA} to denote the entries in a frieze $\frak{F}$.  Each entry in the first nontrivial row, i.e. the row given by the quiddity sequence, comes from an arc whose endpoints are separated by two boundary segments.   More precisely, every such arc is associated to a marked point, the common endpoint of these two boundary segments, which then gives an entry of the quiddity sequence.  An entry $a_{i,j+1}$ is directly to the right of $a_{i,j}$ in $\mathfrak{F}$ whenever the arc $\gamma_{i,j+1}$ can be obtained from $\gamma_{i,j}$ by moving each endpoint of $\gamma_{i,j}$ counterclockwise to the next marked point on the boundary.  Finally, given an entry $a_{1,j}$ in the first nontrivial row, the entries $a_{2,j}, a_{3,j}, \dots, a_{n,j}$ in $\mathfrak{F}$ on the southeast diagonal starting with $a_{1,j}$ correspond to the arcs $\gamma_{2,j}, \gamma_{3,j}, \dots, \gamma_{n,j}$ that are obtained by fixing one endpoint of $\gamma_{1,j}$ and moving the other endpoint counterclockwise to the next marked point one step at a time such that  each time we obtain an arc in the surface.   The category of arcs for a hexagon is depicted in Figure~\ref{typeA-example}.

In type $D_n$, the correspondence between entries in a frieze and arcs in a once-punctured disk is analogous, except for the last two nontrivial rows.  The entries in the last two rows correspond to arcs in the surface that are attached to the puncture or loop around the puncture.  Here, we omit the precise description of these entries, as it will not be important in what follows.  We refer to \cite[Sections 2.3 and 2.5]{S} for additional details.  The category of arcs for a once-punctured disk with four marked points on the boundary is depicted in Figure~\ref{typeD-example}.   

In addition, we extend the correspondences above to include the trivial rows of 1's, as shown in Figures \ref{typeA-example} and \ref{typeD-example}.  In this way, the rows of 1's in a frieze correspond to boundary segments of a disk and a once-punctured disk respectively.


\subsubsection{Cluster character map: associating an integer to each arc}\label{CCmap}

Let $({\bf S},{\bf M})$ be a polygon or a once-punctured disk with a triangulation ${\T}$ and let ${\T}=\{\gamma_1, \dots, \gamma_n\}$. We can associate a quiver $Q$ to ${\T}$ and obtain a corresponding cluster algebra $\mathcal{A}_Q\subset \mathbb{Q}(x_1, \dots, x_n)$, such that the generators of $\mathcal{A}_Q$, called {\it cluster variables}, are in bijection with arcs in the surface \cite[Theorem 7.11]{FST}. 
  
Moreover, there is an explicit combinatorial formula, known as the {\it Laurent expansion formula}, that yields a cluster variable $L_{\gamma}(x_1, \dots, x_n)$ for each arc $\gamma$ in the surface \cite[Theorem 4.9]{MSW}.   We omit the precise expression of the formula, because it is involved, and it will not be important for the rest of the discussion.   We only remark that each $L_{\gamma}(x_1, \dots, x_n)$ is a Laurent polynomial with positive integer coefficients.  The entry in the frieze $\frakF_{\T}$ assigned to $\gamma$ is then obtained via the evaluation of $L_\gamma (x_1, \ldots, x_n )$ at $x_i =1$ for all $i$  \cite[Proposition~5.2]{CCh}.  Therefore, the entry in the frieze equals the number of terms in $L_\gamma (x_1, \ldots, x_n )$ counting multiplicities.  Together with the description of how arcs relate to frieze entries given in Section~\ref{arc_category}, this gives a direct way to compute frieze entries from a triangulation without using the numbers in the quiddity sequence and the diamond rule recurrence.  For an arc $\gamma$ one obtains a certain graph, known as the snake graph, and then counts the number of perfect matchings of this graph (i.e. subsets of the edges of the graph such that each vertex belongs to exactly one edge).  The number of perfect matchings equals corresponding entry in the frieze.  The exact construction of the snake graph is complicated and uses the information of how the arc $\gamma$ crosses other arcs of the fixed starting triangulation $\T$.  For concrete examples that illustrate these computations in detail we refer to lecture notes \cite[Chapter~3]{Sch18}, but we omit the detailed explanations here.

In the Laurent expansion formula, the number of terms in the sum depends on the crossing pattern between $\gamma$ and arcs in $\T$.   The two simplest cases are as follows.  If there is no crossing, that is $\gamma = \gamma_i$ is an arc in $\T$, then the corresponding Laurent polynomial is a single variable $x_i$.  If $\gamma$ crosses a single arc $\gamma_i$ in $\T$ then the Laurent polynomial has the form $\dfrac{x_b x_d + x_c x_e x_f}{x_i}$ (formula also follows explicitly from Ptolemy relations in type $A$ and $D$, see \cite[Section 2.2]{FP}), where some of the variables appearing in the numerator may be 1.

\begin{remark}\label{frieze_entries} 
For us it will be important to know the two facts,  
which are a direct application of the Laurent expansion formula.  
\begin{enumerate}

\item In a frieze $\frakF_{\T}$ of type $A$ or $D$, the entry $a_{i,j}=1$ if and only if the corresponding arc $\gamma_{i,j}$ is in the triangulation $\T$.   

\item In a frieze $\frakF_{\T}$ of type $A$ the entry $a_{i,j}=2$ if and only if the corresponding arc $\gamma_{i,j}$ crosses a unique arc in the triangulation $\T$. 
\end{enumerate}
\end{remark}


\subsection{Mutation of type $A$ friezes}\label{section-mutation-type-A}

It is simplest to state the definition of mutation in terms of triangulations.   A \emph{self-folded} triangle is defined as the {grey} triangle in Figure~\ref{quiddity-sec}.  It consists of two arcs with a common endpoint, where one of the arcs, called the \emph{radius}, goes to the puncture while the other encloses the puncture in a loop.

\begin{defi}\label{def:mut}
Let $\bf{T}$ be a triangulation of $(\bf{S}, \bf{M})$, then for any arc $ a \in \bf{T}$ that is not a radius of a self-folded triangle there exists a unique arc $a'$ such that 
$$\mu_{a}(\mathbf{T}):= \mathbf{T} \setminus\{a \}\cup \{{a'} \}$$ 
is again a triangulation of $(\bf{S}, \bf{M})$.  We call $\mu_{a}(\bf{T})$ the {\it mutation} of $\bf{T}$ at $a$.  For a fixed $a$, we will also denote $\mu_{a}(\bf{T})$ by ${\bf T}'$ to simplify the notation. 
\end{defi}

Mutation of friezes of type $A$ was described in \cite{baur2018mutation}.  Given a frieze $\frakF_{\bf{T}}$, the authors find a formula for the entries in a new frieze $\frakF_{\bf{T}'}$ obtained from $\frakF_{\bf{T}}$ by mutation at an arc $a$.   It relies on subdividing the frieze $\frakF_{\T}$ into regions that depend on the particular configuration of 1's.   For each interior arc $a$ of a triangulation, there is a natural quadrilateral with edges $\{b,c,d,e\}$ that can be associated to $a$, such that the arc $a$ is a diagonal of the quadrilateral (for an example see Figure~\ref{type-A}).

\begin{figure}[h!]
\centering\scalebox{.87}{
\def\svgwidth{1.8in}
\begingroup%
  \makeatletter%
  \providecommand\color[2][]{%
    \errmessage{(Inkscape) Color is used for the text in Inkscape, but the package 'color.sty' is not loaded}%
    \renewcommand\color[2][]{}%
  }%
  \providecommand\transparent[1]{%
    \errmessage{(Inkscape) Transparency is used (non-zero) for the text in Inkscape, but the package 'transparent.sty' is not loaded}%
    \renewcommand\transparent[1]{}%
  }%
  \providecommand\rotatebox[2]{#2}%
  \newcommand*\fsize{\dimexpr\f@size pt\relax}%
  \newcommand*\lineheight[1]{\fontsize{\fsize}{#1\fsize}\selectfont}%
  \ifx\svgwidth\undefined%
    \setlength{\unitlength}{273.38208508bp}%
    \ifx\svgscale\undefined%
      \relax%
    \else%
      \setlength{\unitlength}{\unitlength * \real{\svgscale}}%
    \fi%
  \else%
    \setlength{\unitlength}{\svgwidth}%
  \fi%
  \global\let\svgwidth\undefined%
  \global\let\svgscale\undefined%
  \makeatother%
  \begin{picture}(1,0.99587195)%
    \lineheight{1}%
    \setlength\tabcolsep{0pt}%
    \put(0,0){\includegraphics[width=\unitlength,page=1]{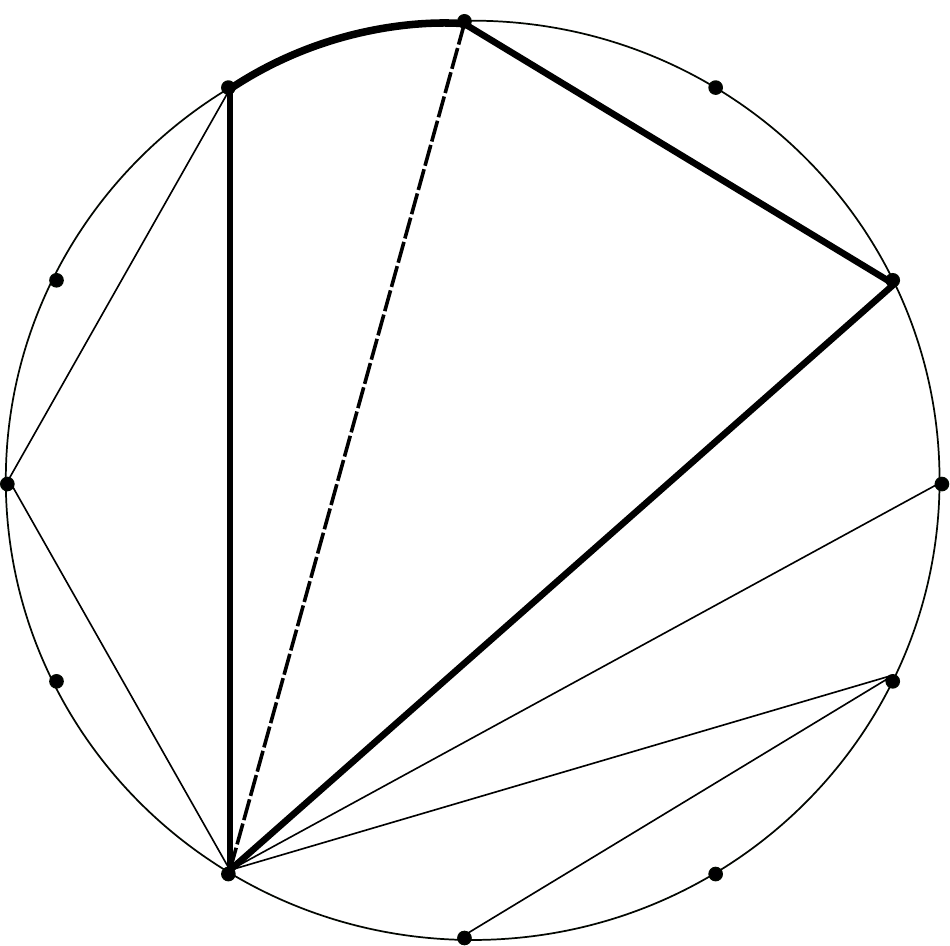}}%
    \put(0.40674103,0.53089105){\color[rgb]{0,0,0}\makebox(0,0)[lt]{\lineheight{1.25}\smash{\begin{tabular}[t]{l}$a$\end{tabular}}}}%
    \put(0.25737828,0.66668299){\color[rgb]{0,0,0}\makebox(0,0)[lt]{\lineheight{1.25}\smash{\begin{tabular}[t]{l}$b$\end{tabular}}}}%
    \put(0.29830663,0.97363887){\color[rgb]{0,0,0}\makebox(0,0)[lt]{\lineheight{1.25}\smash{\begin{tabular}[t]{l}$e$\end{tabular}}}}%
    \put(0.51023871,0.40838437){\color[rgb]{0,0,0}\makebox(0,0)[lt]{\lineheight{1.25}\smash{\begin{tabular}[t]{l}$c$\end{tabular}}}}%
    \put(0.56122656,0.83420626){\color[rgb]{0,0,0}\makebox(0,0)[lt]{\lineheight{1.25}\smash{\begin{tabular}[t]{l}$d$\end{tabular}}}}%
    \put(0,0){\includegraphics[width=\unitlength,page=2]{triangulation-type-A.pdf}}%
    \put(0.59231328,0.71509164){\color[rgb]{0,0,0}\makebox(0,0)[lt]{\lineheight{1.25}\smash{\begin{tabular}[t]{l}$a'$\end{tabular}}}}%
  \end{picture}%
\endgroup%
}
\caption{Local configuration for the arc $a$.  Mutation at $a$ yields a new triangulation with arc $a'$.}
\label{type-A}
\end{figure}

\begin{remark} 
We note that some of $b,c,d,e$ might be on the boundary, and hence might not correspond to arcs in $\bf{T}$, as is the case of the arc $e$ in Figure \ref{type-A}, but  this allows us to define the regions in a uniform way. 
Recall from Section~\ref{arc_category} that the boundary arcs correspond to the rows of 1's. In terms of categorification, they correspond to projective-injective objects in a certain Frobenius category, that can be associated to the disk and contains the cluster category $\mathcal{C}_{Q_{}}$.  We refer to \cite[Example 5.3]{JKS} for a more detailed explanation.  
\end{remark}

\begin{defi}\label{def:pivot}
An {\it elementary move} in the arcs category is defined by moving a single endpoint of an arc $\gamma$ counterclockwise one step while keeping the second one fixed.  
\end{defi}

These moves provides a skeleton for $\frakF_{\T}$ in type $A$, given by arrows $\nearrow, \searrow$ connecting a given entry $m_{\gamma}$ in the frieze with its immediate neighbors located in the rows just above and below $m_{\gamma}$. From now on, when we mention friezes and patterns, the patterns also have the inherited arrow structure. For simplicity, we will not draw said arrows in our figures.

\begin{defi}  Let $m_1, \dots, m_t$ be entries in a frieze $\frakF$ of type $A$.   A path $m_1 \to \cdots \to m_t$ starting at $m_1$ and ending at $m_t$ is called {\it sectional} if all $m_i$ lie in different rows of $\frakF$.  
\end{defi}

Equivalently, a sectional path starting at $m_1$ is either a diagonal path in NE or SE direction.   

\begin{remark}\label{sec_paths}
It follows from the description of type $A$ friezes in Section 2.2.1, that there exists a sectional path starting at $m_1$ and ending at $m_t$ if and only if the corresponding arc $\gamma_{m_t}$ can be obtained from $\gamma_{m_1}$ by moving one endpoint of $\gamma_{m_1}$ counterclockwise while keeping the second endpoint fixed.  
 \end{remark}

By Remark~\ref{frieze_entries} every arc $\gamma \in {\bf T}$ corresponds to a unique entry 1 and a unique entry 2 in the frieze $\frakF_{\T}$ which we denote by $1_{\gamma}$ and $2_{\gamma}$ respectively.   Recall the relative configuration of the arcs $a,b,c,d,e$ in Figure~\ref{type-A}.   Note, the mutation of the arc $a\in \T$ results in a new arc $a'$ that crosses only $a$, so Remark~\ref{frieze_entries}(2) implies that the entry $2_a$ corresponds to the arc $a'$ in the frieze $\frakF_{\T}$. 
By Remark~\ref{sec_paths} the arcs $a,b,c,d,e$  yield sectional paths in $\frakF_{\T}$ starting/ending in 1's or the entry $2_a$ as in  Figure~\ref{fig:regions}.  For example, we can move one endpoint of $a$ counterclockwise while keeping the other one fixed until we obtain arc $d$, which means that there exists a sectional path in $\frakF_{\T}$ starting at $1_a$ and ending in $1_d$.  The existence of other paths follows in a similar way.  These sectional paths subdivide $\frakF_{\T}$ into four disjoint regions $\mathcal{X}, \mathcal{Y}, \mathcal{Z}, \mathcal{F}$ that we define next.

\begin{defi}\label{def_regionsA} (Regions for type $A$)

\begin{enumerate}
\item Let $\calX$ be the region in $\frakF_{\T}$ consisting of entries in the interior of two rectangles determined by the sets of vertices $\{1_b, 2_a, 1_c\}$ and $\{1_d, 2_a, 1_e\}$. 
\item Let $\calY$ be the region in $\frakF_{\T}$ consisting of entries in the interior of two rectangles determined by the sets of vertices $\{1_e, 1_a, 1_b\}$ and $\{1_c, 1_a, 1_d\}$.
\item Let $\calZ$ be the region in $\frakF_{\T}$ consisting of entries in the interior of the two rectangles with vertices $\{1_a, 1_b, 1_d, 2_a\}$ and $\{2_a, 1_e, 1_c, 1_a\}$.   The region $\overline{\mathcal{Z}}$ is defined as $\mathcal{Z}$ together with the entries on the neighboring sectional paths except for $1_x$ with $x \in\{b,c,d,e\}$. 
\end{enumerate}

\end{defi}

In particular, note that $1_a, 2_a \in \overline{\mathcal{Z}}$.  We also define the region $\calF$ that contains all the remaining entries in $\frakF_{\T}$.

Before stating the theorem, we recall certain projections associated to entries in various regions.

\begin{figure}[h!]
\centering
\scalebox{.95}{\def\svgwidth{6.8in}
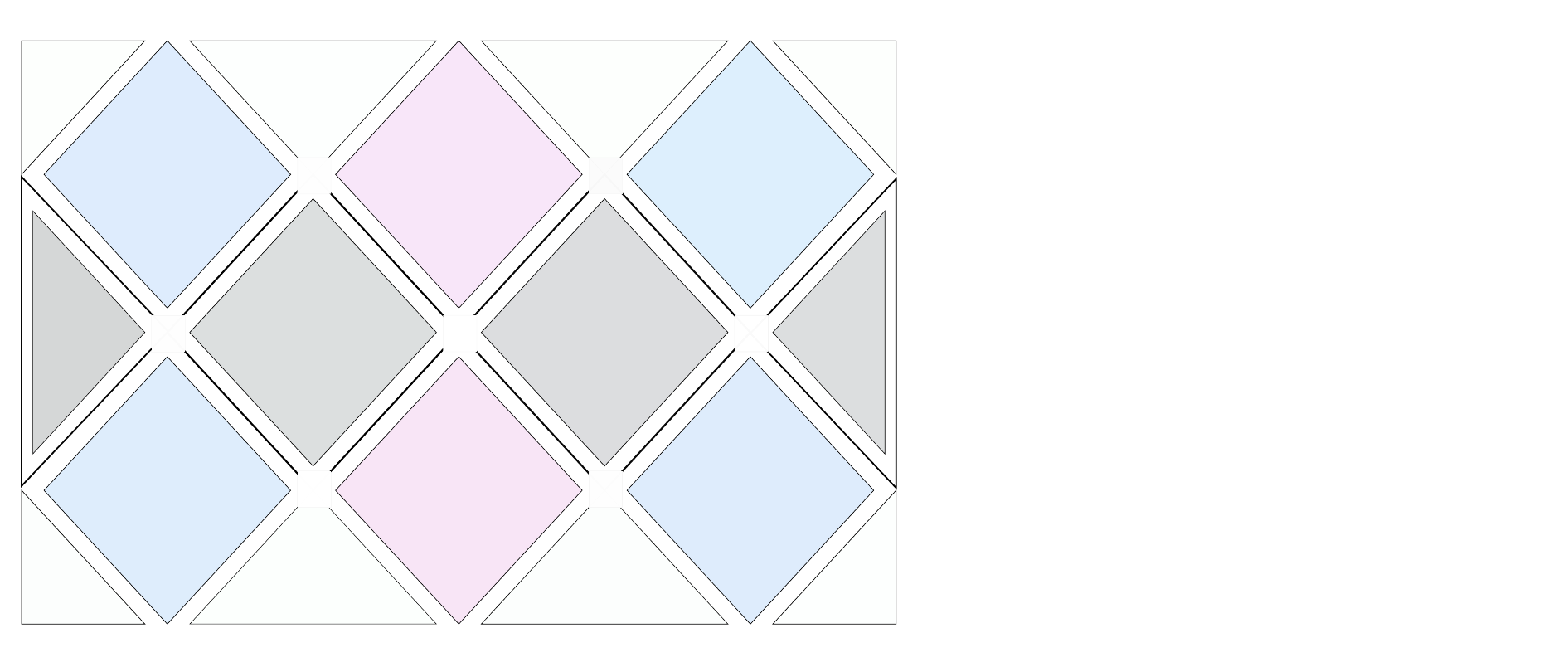}
\caption{Regions in $\frakF_{\T}$ defined by mutation at $a$ and projections for $m \in \calX$ and $m \in \calZ$.}
\label{fig:regions}
\end{figure}

Now we define several projections for a frieze of type $A$, needed to state the last result of this section.  Later, there will be analogue projections for type $D$.  
 
 \begin{defi}\label{projections-A} (Projections for type $A$)
 \begin{enumerate}
 \item Given $m$ in $\calX$ or $\calY$ there exist unique sectional paths  $\gamma^+, \gamma^-$  starting and ending at $m$ respectively, such that $\gamma^+$ and $\gamma^-$ intersect two other sectional paths: one passing through $1_a$ and the other passing through $2_a$. The projection of $m$ onto the closest of these two paths along $\gamma^+$ is denoted by $\pi^+_1(m)$ and the projection onto the second path is $ \pi^+_2(m)$. The projection of $m$ onto the closest path along $\gamma^-$ is denoted by $\pi^-_1(m)$ and onto the second path it is denoted by $\pi^-_2(m)$.  
 
 \item Given $m\in \overline{\calZ}$ we define four projections for $m$ onto the four boundary edges of $\overline{\calZ}$.  We denote the projections onto the paths starting or ending at $1_a$ by $\pi_p^\uparrow(m), \pi_p^\downarrow(m)$ and the projections onto the paths starting or ending at $2_a$ by $\pi_s^\uparrow(m), \pi_s^\downarrow(m)$ respectively. We choose the upwards arrow to refer to the paths ending/starting near $1_b$ or $1_c$ and the downwards arrow to refer to paths ending/starting near $1_d$ or $1_e$.
 
\end{enumerate}  
\end{defi}

See the projections in Figure~\ref{fig:regions}.

\begin{theorem} \cite[Theorem 6.12]{baur2018mutation} \label{thmA}
Let $m\in \frakF_{\T}$, and $m'\in \frakF_{\T'}$ be the corresponding integer after mutation at vertex $a$.  Then $\delta_a (m):= m-m'$ is computed as follows.

\begin{itemize}\setlength\itemsep{4pt}
\item If $m \in \calX$ then $\delta_a (m) = [\pi_1^+(m) - \pi_2^+(m)] [\pi_1^-(m) - \pi_2^-(m)]$.
\item If $m\in \calY$ then $\delta_a (m) = -[\pi_2^+(m) - 2 \pi_1^+(m)] [\pi_2^-(m)- 2 \pi_1^-(m)]$.
\item If $m \in \overline{\calZ}$ then $\delta_a (m) = \pi_s^\downarrow (m) \pi_p^\downarrow(m) + \pi_s^\uparrow(m) \pi_p^\uparrow(m) - 3 \pi_p^\downarrow(m) \pi_p^\uparrow(m)$. 
\item If $m\in\calF$ then $\delta_a(m)=0$, i.e. $m$ does not change.
\end{itemize}
\end{theorem}


\section{From type $D$ to type $A$} \label{from type D}

In this section, we develop a construction that allows us to convert a type $D$ frieze into a type $A$ and avoid various complications that arise from the particular structure of type $D$ friezes. Moreover, we show that this transformation behaves well under mutations.  

\subsection{Cutting an arc at the puncture}\label{section-cut} Let $\bf{T}$ be a triangulation of a once-punctured disk $\bf{S}$ with $n$ marked points on the boundary.   Next, we describe a new surface obtained from $\bf{S}$ by cutting along an arc going to the puncture $p$.  Note that there is always at least one arc in $\T$ attached to the puncture.

\begin{defi}\label{def:cut}
Let $({\bf S}, \T)$ be a triangulated disk with $n$ marked points on the boundary and one puncture $p$.  Let $i \in \bf{T}$ be an arc that is attached to the puncture.   The surface $({\bf S}_i, {\T}_i)$ is a triangulated polygon defined by cutting $\bf{S}$ along $i$ as in Figure~\ref{cut-surface}.  The polygon ${\bf S}_i$ has the same boundary segments as $\bf{S}$ and two additional ones $i_1, i_2$, where $i_2$ appears clockwise from $i_1$ and they meet at $p$. 
\end{defi}

Next we define a new triangulated surface $(\tilde{\bf S}_i, \tilde{\bf T})$.

\begin{figure}[h!]
\centering
\scalebox{.88}{\def\svgwidth{6in}
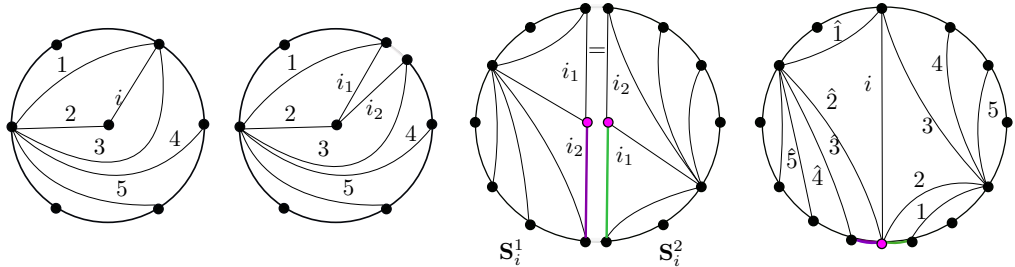}
\caption{Triangulation of the disk $\tilde{\bf S}_i$ obtained after cutting the triangulated punctured disk at $i$.}
\label{cut-surface}
\end{figure}

\begin{defi}\label{def:surface}
Let $({\bf S}_i^1, {\bf T}_i^1), ({\bf S}_i^2, {\bf T}_i^2)$ be two copies of the triangulated surface ${\bf S}_i$ as in Definition~\ref{def:cut}, where the first one is obtained from ${\Sur}_i$ after rotating $i_2$ clockwise until we get a half disk, while the second one is obtained after rotating $i_1$ counterclockwise. We define $(\tilde{\bf S}_i, \tilde{{\bf T}})$ as the triangulated  surface obtained after gluing along the boundary segment $i_1$ in ${\bf S}_i^1$ the corresponding boundary segment $i_2$ in ${\bf S}_i^2$ (see Figure \ref{cut-surface}) in such way that the endpoint $p$ in $i_1$ is glued to the endpoint $p$ in $i_2$.  
\end{defi}

Observe that the surface $\tilde{\bf S}_i$ is (topologically) an unpunctured disk with $2n+2$ marked points on the boundary.  The unglued arcs $i_1,i_2$ become boundary segments in $\tilde{\bf S}_i$. The triangulation $\tilde{\bf T}$ has $2n-1$ arcs that are inherited from the triangulations ${\T}_i^1, {\T}_i^2$ of  ${\bf S}_i^1, {\bf S}_i^2$ respectively.  Given an arc $\gamma$ in ${\T}_i^1\setminus\{i\}$ or ${\T}_i^2\setminus\{i\}$ we label the corresponding arc in $\tilde{\bf T}$ by $\hat{\gamma}, \gamma$ respectively.  Let $i$ denote the remaining arc of $\tilde{\bf T}$ that arises as a diameter of the resulting disk in the construction. We always draw $\tilde{\bf S}_i$ such that $p$ appears in the bottom,  see Figure \ref{cut-surface} (right).

Note that if $i$ is a radius of a self-folded triangle in $({\bf S}, {\bf T})$, with the associated loop $\gamma(i)$, then after cutting ${\bf S}$ along $i$ the self-folded triangle becomes a triangle in ${\bf S}_i$ with boundary edges $i_1, i_2$ and the edge $\gamma(i)$. Then by construction of $(\tilde{\bf S}_i, \tilde{{\bf T}})$, we obtain two triangles in $\tilde{\bf S}_i$ symmetric with respect to $i$ with edges $i, \gamma(i), i_2$ and $i, \hat{\gamma(i)}, i_1$ where $i_1, i_2$ are the unglued boundary segments.


\subsection{Glued type $D$ pattern} Next, we show how a frieze of type $D_n$ coming from a triangulation $\bf{T}$ of the punctured disk can be embedded into a frieze of type $A_{2n-1}$ coming from $\tilde{\bf T}$.  Recall that a frieze of type $D_n$ has the same structure as type $A_n$ except for the last two rows, see Definitions \ref{def:A} and \ref{def:D}. 

\begin{defi}\label{def:gluedfrieze}
Given a frieze $\frakF$ of type $D_n$ we construct a \emph{glued type $D$ pattern} $g(\frakF)$ as follows.  Take the first $n$ rows of $\frakF$ (i.e. $n-2$ nontrivial rows) and replace the remaining last two nontrivial rows with a single row, such that the entries of the form  $\begin{smallmatrix}a_{n-1,j}\\a_{n,j}\end{smallmatrix}$  become a single entry in $g(\frakF)$ by taking the product $a_{n,j} a_{n-1,j}$.
\end{defi}

Note that $g(\frakF)$ has the same underlying shape as a frieze of type $A$, and the diamond rule still holds except in the lower rows.  Recall that $g(\frakF)$ does not end with a row of 1's, therefore in general \[(a_{n,j} a_{n-1,j})(a_{n-1, j+1}a_{n,j+1})-a_{n-2, j+1}\not=1.\]

Let $\frakF$ be a frieze of type $D_n$ coming from a triangulation ${\T}$.   By Remark~\ref{frieze_entries}(1), an arc $i\in \T$ attached to the puncture corresponds to the entry $1_i$ in the frieze $\frakF$.  Moreover, the description of the arc category in type $D$ given in Section~\ref{arc_category} implies that $1_i$ lies in one of the bottom two rows of the frieze.  This observation allows us to recover the frieze $\frakF$ from the glued pattern $g(\frakF)$ as follows. 

\begin{defi}\label{def:unglue}
Let $a_{n,j} a_{n-1,j}$ be an entry in the last row of the glued type $D$ pattern $g(\frakF)$ such that one of the factors equals 1.  Choose one of these factors, say $a_{n,j}$, and set it equal to 1.  Then the local configuration in the last three rows of $\frakF$ is
\[\begin{smallmatrix}a_{n-2, j-1}&&a_{n-2,j}\\ &a_{n-1,j-1}&& a_{n-1,j}\\ &a_{n,j-1}&& 1=a_{n,j}\end{smallmatrix}\]
where we consider indices modulo $n$.   The frieze relations in Definition~\ref{def:D} imply that $a_{n,j-1}= 1 + a_{n-2,j}$ and $a_{n-1,j-1}=(1+a_{n-2,j})/a_{n-1,j}$.   Now, given the entries $a_{n,j-1}, a_{n-1, j-1}$ and $a_{n-2, j-1}$ we can solve for the next two entries in $\frakF$ to the left of $a_{n,j-1}, a_{n-1, j-1}$ in the same way.  Continuing this procedure we recover the entire frieze $\frakF$ (up to interchanging the last two rows) from the glued pattern $g(\frakF)$.  We call this process \emph{ungluing}. 
 \end{defi}

The reader can see an example of type $D$ frieze in Figure \ref{frieze-ex} (left) and the corresponding glued pattern in Figure \ref{final-example}.

\begin{defi}\label{def:bcd}
Let ${\T}$ be a triangulation of a once-punctured disk with $n$ marked points on the boundary and puncture $p$.   Let $i\in \T$ be an arc attached to the puncture.   
Let $d$ and $a_i$ denote the number of triangles in ${\T}$ incident with the two endpoints of $i$, where $d$ denotes the number of triangles incident with $p$ and $a_1$ denotes the number of triangles incident with the endpoint of $i$ that lies on the boundary.
Also, consider the corresponding triangulated surface $({\bf S}_i, {\T}_i)$ with boundary arcs $i_1, i_2$ (see Definition~\ref{def:cut}).  Let $b,c$ denote the number of triangles in ${\T}_i$ with vertices being the endpoints of $i_1,i_2$ respectively that are not attached to $p$ (see Figure~\ref{cut-surface-quid}).
\end{defi}

\begin{defi}\label{def:friezeT}
Let $(\tilde{\Sur}_i, {\tT})$ be the triangulated surface given in Definition~\ref{def:surface}.  Then we define $\frakF_{\tT}$ to be the frieze of type $A_{2n-1}$ coming from the triangulation $\tT$ of the disk $\tilde{\Sur}_i$.
\end{defi}

The next proposition relates the frieze $\frakF_{\T}$ of type $D_n$ with the corresponding frieze $\frakF_{\tT}$ of type $A_{2n-1}$.

\begin{prop}\label{prop-glue}
Let $\bf{T}$ be a triangulation of a once-punctured disk $\Sur$ and $\tT$ the associated triangulation on the disk $\tilde{\Sur}_i$.  Then the glued type $D_n$ pattern $g(\frakF_{\T})$ appears as a connected region in the type $A_{2n-1}$ frieze $\frakF_{\tT}$.
\end{prop}
 
\begin{proof}
The entries in a glued type $D_n$ pattern are completely determined by the quiddity sequence coming from $\T$, and similarly the entries in $\frakF_{\tilde{\T}}$ are determined by the sequence coming from $\tT$.  To prove the proposition, we find a precise relationship between the two sequences.  Recall that the entries in a quiddity sequence equal the number of triangles in a triangulation incident to each marked point on the boundary. 

Let $a_1, d, b,c$ be the integers given in the Definition~\ref{def:bcd}.  By Definition~\ref{def:cut} the puncture $p$ becomes a boundary point in ${\bf S}_i$ and its corresponding entry in the quiddity sequence equals $d$.   For example, see Figure~\ref{cut-surface-quid}.  Thus, given a quiddity sequence $(a_1, a_2, \dots, a_n)$ for $\Sur$ we obtain the quiddity sequence $(b, a_2, \dots, a_n, c, d)$  for ${\bf S}_i$.  By construction in Definition~\ref{def:surface}, the polygon $\tilde{\bf S}_i$ is obtained by gluing two copies of ${\bf S}_i$ at $i_1$ and $i_2$.  The resulting quiddity sequence for $\tilde{\bf S}_i$ has length $2n+2$ and can be written as follows. 

$$q=(2d, b, a_2, \dots, a_n, a_1, a_2, \dots, a_n, c)$$
\begin{figure}[h!]
\centering
\scalebox{.85}{\def\svgwidth{5.6in}
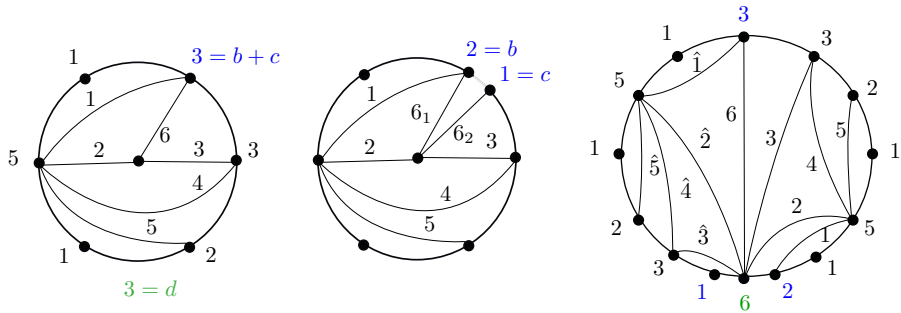}
\caption{The quiddity sequence $(6,2,1,5,1,2,3,3,1,5,1,2,3,1)$ of type $A$ coming from the sequence $(3,1,5,1,2,3)$ of type $D$. } 
\label{cut-surface-quid}
\end{figure}

In the frieze $\frakF_{\tT}$, the entries in the first $n-1$ nontrivial rows determined by the subsequence 
$$q'=(a_2, \dots, a_n, a_1, a_2, \dots, a_n)$$ 
of $q$ contains $g(\frakF_{\T})$.  See the diagram below. 

$$\xymatrix@C=0pt@R=0pt{b \,\,\,\ar@{-}[dddrrr] &a_2 &&\dots && a_n & \ar@{--}[dddll] \,\,\, a_1& a_2 && \dots && a_n & \ar@{-}[dddll]\,\,\, \, c& 2d\\
& \\
&&&&&&&&g(\frakF_{\T})\\
&&&\ar@{-}[rrrrrrr] &&&&&& &}$$

Indeed, this glued type $D$ pattern has the same quiddity sequence, and since $q'$ has length $2n-1$ the $(n-1)$-st nontrivial row of $\frakF_{\tT}$ determined by this subsequence has $n+1$ entries.  The period of $g(\frakF_{\T})$ is $n$, therefore it is contained in this region as claimed.    
\end{proof}

 

Next, we show that this construction behaves well with mutations.

\begin{defi}\label{def:mutfrieze}
Let $\T$ be a triangulation of a disk (or a once-punctured disk) ${\bf S}$ with the associated type $A_n$ (or a type $D_n$) frieze $\frakF_{\T}$.   For an arc $a$ in $\T$, let $\mu_a({\T})$ denote the new triangulation of ${\bf S}$ obtained by mutating $\T$ at $a$ (see Definition~\ref{def:mut}).  Define the {\it mutation of $\frakF_{\T}$ at $a$} to be the new frieze $\mu_a(\frakF_{\T}):=\frakF_{\mu_a ({\T})}$ of type $A_n$ (or $D_n$) coming from the triangulation $\mu_a({\T})$.
\end{defi}

Let $\T$ be a triangulation of the once-punctured disk ${\bf S}$ with an arc $i\in \T$ attached to the puncture.  Let $a$ be any other arc of $\T$ different from $i$.  Consider the associated disk $\tilde{\bf S}_i$ with triangulation $\tT$.  By Definition~\ref{def:surface} the triangulation $\tT$ contains two arcs $a,\hat{a}$ coming from the arc $a\in \T$ with the same label.  By Proposition~\ref{prop-glue} the glued pattern $g(\frakF_{\T})$ appears as a connected region in $\frakF_{\tT}$, and   
let $\mu_a\mu_{\hat{a}}(g(\frakF_{\T}))$ denote the sub-pattern of $\mu_a \mu_{\hat{a}} (\frakF_{\tT})$ occupying the position of   $g(\frakF_{\T})$ after the mutations at $a$ and $\hat{a}$.

\begin{theorem}\label{mut-thm}
Let $\bf{T}$ be a triangulation of a once-punctured disk $\bf{S}$ and let $\tilde{\bf T}$ be the corresponding triangulation of a polygon $\tilde{\bf S}_i$ for some $i$.  Then for all arcs $a\not= i$ in $\bf{T}$ 
\begin{itemize}
\item[(a)] $\widetilde{\mu_a({\bf T})} = \mu_a\mu_{\hat{a}}(\tT)$
\item[(b)] $\mu_a\mu_{\hat{a}}(g(\frakF_{\T})) = g(\frakF_{\mu_a({\T})})$.
\end{itemize}
In particular, $\mu_a (\frakF_{\T})$ is obtained by ungluing  the sub-pattern $g(\frakF_{\mu_a(\T)})$ contained in $\mu_a \mu_{\hat{a}} (\frakF_{\tT})$.
\end{theorem}

\begin{proof}
By construction in Definition~\ref{def:surface}, the polygon $\tilde{\bf S}_i$ is obtained by gluing two copies of $\bf{S}$ at $i$.  Since $i \not=a$, the mutations of $\tilde{\bf T}$ at $a$ and $\hat{a}$ occur in their appropriate copies of $\bf{S}$, and they do not interfere with each other.  In particular, we obtain the same triangulation of $\tilde{\bf S}_i$ by first mutating $\bf{T}$ at $a$ in the polygon $\bf{S}$ and then gluing two resulting copies together.  This shows part (a). 

To prove part (b), we identify $g(\frakF_{\bf T})$ with the region in $\frakF_{\tT}$ as in the proof of Proposition~\ref{prop-glue}.  In particular, $g(\frakF_{\T})$ lies in the portion of $\frakF_{\tT}$ determined by the quiddity subsequence $(a_2, \dots, a_n, a_1, \dots, a_n)$ coming from $\tT$.  The mutations $\mu_a, \mu_{\hat{a}}$ change the quiddity sequence to some $q''$ and yield $\mu_a\mu_{\hat{a}}(g(\frakF_{\T}))$.  

Moreover, $q''$ is uniquely determined by the triangulation $\mu_a\mu_{\hat{a}}(\tT)$.  By part (a) of the theorem we have $\mu_a\mu_{\hat{a}}(\tT)=\widetilde{\mu_a({\T})}$, which implies that $\mu_a\mu_{\hat{a}}(g(\frakF_{\T})) = g(\frakF_{\mu_a(\T)})$ as desired.  This shows part (b). 

Note that $1_i$, the entry 1 corresponding to the arc $i\in {\T}$ as in Remark~\ref{frieze_entries}(1), lies in one of the bottom two rows of $\frakF_{\T}$.  By Definition~\ref{def:mutfrieze} we have $\mu_a(\frakF_{\bf T})= \frakF_{\mu_a(\bf T)}$, and the final statement follows from part(b) and the Definition~\ref{def:unglue} of ungluing.  Observe that $1_i$ remains in the same position in both $\frakF_{\bf T}$ and $\mu_a(\frakF_{\bf T})$ enabling to unglue $g(\frakF_{\mu_a (\T)})$.
\end{proof}

The theorem allows us to study mutations of friezes of type $D$ as pairs of mutations in type $A$.  Moreover, the particular type $A$ friezes introduced in Definition \ref{def:friezeT} possess additional symmetries that we describe in the next section.  

\begin{remark} A similar notion of gluing appears in \cite{Ma} where the author studies relationships between cluster variables of types $A$ and $D$.  However, that construction relies on the notion of slices and it is done only on the level of Auslander--Reiten quivers.  For us it is important to understand this transformation in terms of surfaces to study symmetries in the resulting type $A$ frieze.  Our construction corresponds to taking a particular orientation of the slice $\overline{\omega}$ in \cite[Section 2.3]{Ma}.
 
On the other hand, Amiot--Plamondon \cite{AP} have used geometric moves resembling the ones in Figure~\ref{cut-surface-quid} to describe cluster categories arising form punctured Riemann surfaces in terms of cluster categories arising from unpunctured surfaces and together with a group action. Let us mention that their construction is different. For example, a punctured disk with $n$ marked points on the boundary would be transformed into a disk with $2n-3$ marked points on the boundary. Moreover, their construction applies only in the case when all punctures are enclosed by self-folded triangles.
\end{remark}


\section{Mutation of friezes of type $D$} \label{mut-D}

In this section we describe mutations of friezes of type $D$.  Instead of working with the frieze directly, we exploit the symmetries present in $\frakF_{\tT}$, and we define an auxiliary pattern $\frakG_{\T}$ (see Definition~\ref{def:patternG}) that behaves like a frieze of type $A$ under mutations. We will see that our glued type $D$ pattern $g(\frakF_{\T})$ can be thought as a full sub-pattern of $\frakG_{\T}$.

\subsection{Properties of $\frakF_{\tT}$ and the pattern $\frakG_{\T}$}

We now describe the important regions $\calA,\calB, \calC,\hat{\calA}$ of the glued pattern $g(\frakF_{\T})$ (Definition~\ref{def:gluedfrieze}) and the frieze $\frakF_{\tT}$ (Definition \ref{def:friezeT}), see Figure~\ref{AR-quiver-1}.  Also, recall the entries $a_1, b, c, d$ given in Definition~\ref{def:bcd}.  Then the glued pattern $g(\frakF_{\T})$ is the disjoint union of two triangular regions $\calA \cup \calB$ defined as follows. The top row of $\calA$  is a part of the quiddity subsequence $(a_2, \ldots, a_n)$ of type $D_n$.  The triangles $\calB$ and $\calA$  have $n-1$ rows. The triangle $\calB$ is inverted with respect to $\calA$, its top row is the remaining element $a_1$.  The frieze $\frakF_{\tT}$ is then completed by another copy of $\calA$, that we denote by $\hat\calA$, a triangular region $\calC$ with entry $2d$ in the top row, 
and a sectional path $\mathfrak{s}$ starting at $b=r_1$ and ending at $\hat r_{n-1}$.

\begin{figure}[h!]
\centering
\scalebox{.9}{\def\svgwidth{4.8in}
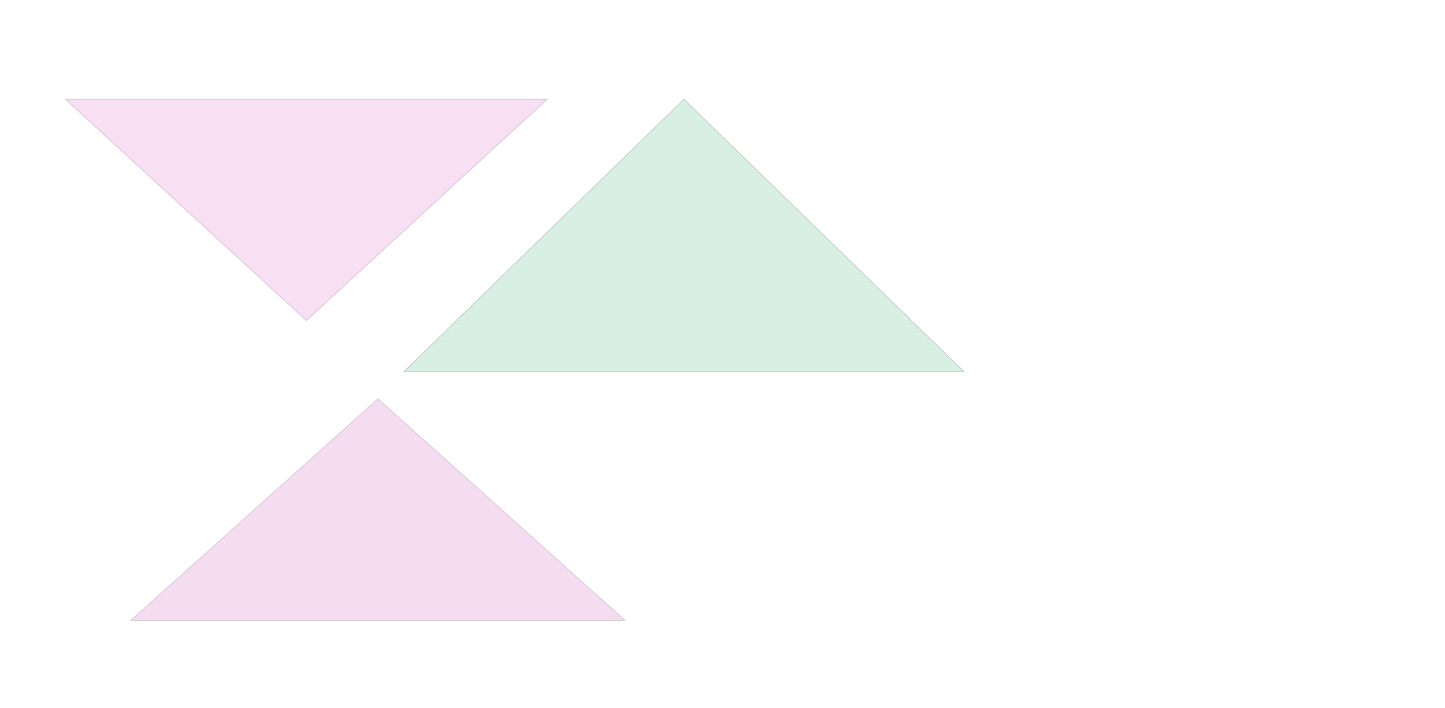}
\caption{Frieze $\frakF_{\tT}$ that contains the glued type $D$ pattern $\calA \cup \calB = g(\frakF_{\T})$. }
\label{AR-quiver-1}
\end{figure}

Consider Figure~\ref{R} (a) that depicts the triangulation $\tT$ of $\tilde{\bf S}_i$.
The sectional path $\mathfrak{s}$ corresponds to performing elementary moves at the point $p$ (see Definition~\ref{def:pivot} and Remark~\ref{sec_paths}), where $p$ comes from the puncture $p$ after the construction of $\tilde{\bf S}_i$. The first nontrivial entry in $\mathfrak{s}$ is $b=r_1$, associated to the arc $\sigma_1$, and all other nontrivial entries are obtained by moving the endpoint $u$ counterclockwise until $v$ is reached. This last integer is $\hat r_{n-1}=c$.   By Remark~\ref{frieze_entries}(1), the entry $1_i$ in the frieze $\frakF_{\tT}$ corresponds to the arc $i\in \tT$, which implies that the entry $1_i$ appears at the center of the sectional path $\mathfrak{s}$.

\begin{figure}[h!]
\centering
\scalebox{.82}{\def\svgwidth{6in}
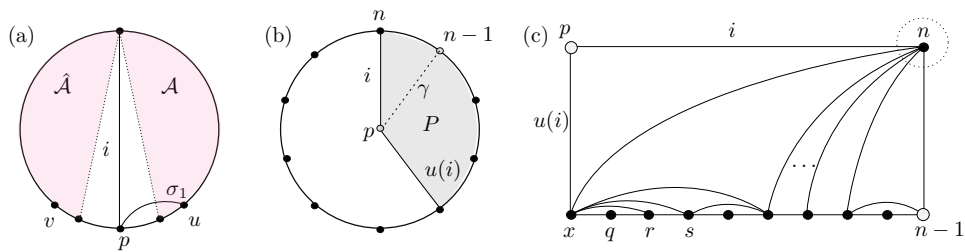}
\caption{(a) Arcs defining sectional path $\mathfrak{s}$; (b) arc $\gamma$; (c) Baur-Marsh definition of the integer associated to $\gamma$.}
\label{R}
\end{figure}

Recall from Section~\ref{CCmap} that an entry in a frieze $\frakF_{\T}$ associated to an arc $\gamma$ depends on the crossing pattern of $\gamma$ and arcs in the triangulation ${\T}$.   In \cite{BM} the authors provide a purely combinatorial description of the entries in the frieze. Different types of arcs in the surface have different combinatorial rules of how to calculate the associated entries in a frieze, for example there is a Laurent expansion formula as in Section \ref{CCmap}, see \cite[Section 5.5]{Sch18}.  Below we explain only one of these rules in the case when an arc is attached to the puncture.
Following \cite[Definition 2.14]{BM} the integer in $\frakF_{\T}$ corresponding to the particular arc $\gamma$, shown in Figure~\ref{R}(b), is given by the cardinality of a set denoted by $\mathcal{M}_{n-1,0}$. This set consists of \emph{matchings}, i.e. ways to allocate triangles in the triangulated polygon $P$ to the black marked points on $\partial P$.  Moreover, we can associate a triangle to a black point only if the point is a vertex of this triangle.  Recall the labeling of entries in $\frakF_{\T}$ given in Figure~\ref{friezeD}.

\begin{lemma}\label{lemma-b-and-c}
Let $1_i = a_{n,n}$ be the integer associated to the arc $i$ in a type $D_n$ frieze $\frakF_{\T}$. Then $b=a_{n-1,1}$, $c=a_{n-1,n-1}$ in the case $n$ is even, and $c=a_{n,n-1}$ in the case $n$ is odd.
\end{lemma}

\begin{proof}
We consider the case when $n$ is even and find $a_{n-1,n-1}$, and the remaining cases follow similarly.  Let $\gamma$ be the arc in the punctured polygon corresponding to the entry $a_{n-1,n-1}$ in the frieze. 
Note that if $\gamma \in \T$ then $c=a_{n-1,n-1}=1$, as claimed.  Now, suppose that the arc $\gamma$ in Figure \ref{R} (b) is not in $\T$. Let $u(i)$ be the arc at the puncture in $\T$ immediately clockwise from $i$. Let $P$ be the (triangulated) gray polygon with $t$ marked points on $\partial P$ in Figure \ref{R} (b).  Both endpoints of $\gamma$ are marked white, the other points on $\partial P$ are marked black. Thus, $P$ consists of $t-2$ triangles and $t-2$ black points. 

We draw the triangulated polygon $P$, as in Figure \ref{R} (c). The triangles whose vertices are on the lower edge of the rectangle have only one possible associated black point. For example, the triangle with vertices $x,q,r$ can only be assignable to $q$, since $q$ has to be assigned to a triangle containing $q$. From this, the black point $r$ can be assigned only to the triangle with vertices $x,r,s$, because the other possible triangle with endpoints $x,q,r$ is no longer available. Continuing in this way, we conclude that $\vert \mathcal{M}_{n-1,0} \vert$ depends only on the arcs at the black point $n$.  After one of the triangles having $n$ as a vertex is assigned to $n$, the rest of the correspondences are uniquely determined. Hence, $a_{n-1,n-1}$ is obtained by counting the triangles at $n$, which coincides with the Definition~\ref{def:bcd} of $c$. 
This shows the statement of the lemma. 
\end{proof}

\begin{defi}\label{def:R}
Define two sectional paths $\calR$ and $\hat \calR$ in the frieze $\frakF_{\tT}$ as the following subpaths of 
$$\mathfrak{s}: r_1\to \dots \to r_{n-1}\to 1_i \to \hat r_1 \to \dots \to \hat r_{n-1}.$$  Let $\calR$ be the sectional path $r_1 \to \cdots \to  r_{n-1}$  starting in $r_1$ and ending in $r_{n-1}$, and similarly let $\hat{\calR}$ be the sectional path $\hat r_1 \to \cdots \to \hat r_{n-1}$ starting in $\hat r_1$ and ending in $\hat r_{n-1}$. See Figure~\ref{AR-quiver-1}. 
\end{defi}

In the next lemma we show that in fact $r_j = \hat r_{j}$ for all $j \in [1,n-1]$, and we explain that these values arise from the original type $D_n$ frieze $\frakF_{\T}$.  Recall the labeling of the entries in a frieze of type $D_n$ given in Figure~\ref{friezeD}.

\begin{lemma}\label{lemma-R}
Let $1_i = a_{n,n}$ be the integer associated to the cut arc $i$ in a type $D_n$ frieze $\frakF_{\T}$. Then $ \calR = \hat \calR= a_{n-1, 1}\to a_{n, 2}\to a_{n-1,3} \to \dots \to r_{n-1} $ where \[r_{n-1} = c = \left\{ \begin{array}{lcc}
             a_{n,n-1} &   \mathrm{if} \ n  \ \mathrm{odd} \\
             \\ a_{n-1,n-1} &  \mathrm{if} \ n  \ \mathrm{even} \\
             \end{array}
   \right.\]

and $a_{n-1,1} = b$, i.e. the integers come from the arcs attached to the puncture. 
\end{lemma}

\begin{proof} We explain the case when $n$ is even, and the case when $n$ is odd uses the same argument.  
Claim: The frieze $\frakF_{\tT}$ contains a diamond \[\begin{matrix}&a_{n-2,n}\\r_{n-1}&& a_{n-1,n} a_{n,n}\\& 1_i\end{matrix}\] where $a_{n,n}= 1$. Proof of the claim: We observe that $a_{n-1, n}a_{n,n}$ is the entry in the corner of region $\mathcal{A}$ because of the structure of the frieze $\frakF_{\T}$ and $g(\frakF_{\T})$.  Since $a_{n,n}=1_i$, the entry in the first row of $g(\frakF_{\T})$ that lies on the northeast diagonal starting at $a_{n-1, n}a_{n,n}$ ends in $a_{1, n-2}$.  Moreover, we see that $a_{1,n-2}$ equals $a_n$ in the corresponding quiddity sequence.  This shows the claim.

\begin{figure}[h!]
\centering
\scalebox{.85}{\def\svgwidth{4.6in}
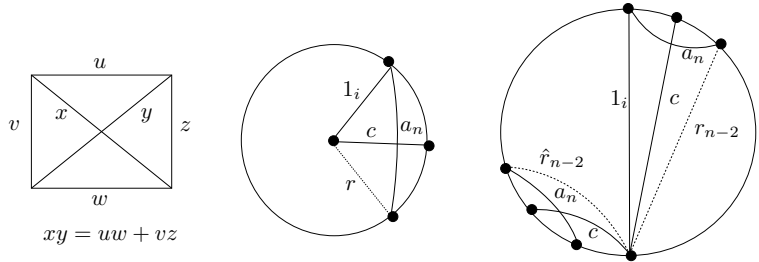}
\caption{Finding $r_{n-2} \in \calR$ via Ptolemy relations.}
\label{untagged}
\end{figure}

From the Definition~\ref{def:A} of type $A$ freize $r_{n-1} a_{n-1,n} = a_{n-2,n} +1$, and from the Definition~\ref{def:D} of type $D$ frieze $a_{n-1,n-1} a_{n-1,n}= a_{n-2,n}+1$. Hence $r_{n-1} = a_{n-1,n-1} = c = \hat r_{n-1}$, where the second equality follows from Lemma \ref{lemma-b-and-c}. 

Now we use the well-known Ptolemy relation for cluster variables in a triangulation, and the fact that our integers arise as evaluations $x_j = 1$ for all $j$ on these variables. See for example \cite[Section 3]{P05} and references therein.
Whenever there is a square such that five integers associated to its edges and diagonals are known, the remaining integer can be obtained from the formula in Figure~\ref{untagged} on the left. In the middle diagram of the same figure we have $r= a_{n,n-2}$ arising from a square (with edges evaluated to) $1,1,1,r$ and diagonals $a_n$ and $c=a_{n-1,n-1}$, where $a_n$ is obtained from the quiddity sequence. The same square evaluation appears in a pair of squares on the unpunctured disk. It follows immediately that $a_{n,n-2}=r_{n-2}=\hat r_{n-2}$. Moving clockwise along the punctured disk, the same idea can be used to find all integers in $\calR$. The last Ptolemy relation arises from a square $1,1,1,r_2=\hat r_2$ with diagonals $\hat r_1=r_1=b$ and $a_2$.   
\end{proof}

Before stating the next lemma we recall Remark~\ref{frieze_entries} that describes the correspondence between entries $1_j,2_j$ in the frieze $\frakF_{\tT}$, indexed by the arcs $j\in \tT$, and arcs in the surface $\tilde{\bf S}_i$ that are in $\tT$ or cross exactly one arc in $\tT$.

\begin{lemma}\label{lemma-A} Let $\tT$ be a triangulation of the disk obtained from a triangulation $\T$ of the punctured disk by a cut at $i$, and let $\frakF_{\tT}$ the associated frieze of type $A$. Then for all $j \in {\T}\setminus\{i\}$
 \begin{enumerate}
 \item[(a)] $1_j \in \calR \cup \calA$, and
 \item[(b)] $2_j \in \calR \cup \calA$.
 \end{enumerate}
\end{lemma}
\begin{proof}
Part (a) follows by the construction of $(\tilde{\bf S}_i, {\tT})$ in Definition~\ref{def:surface} and the category of arcs in types $A$ and $D$, see Subsection \ref{arc_category}.   That is, if $j$ is an arc at the puncture in $\T$ then the corresponding arc in the new triangulation $\tT$ will be part of $\calR$ by Lemma~\ref{lemma-R}. On the other hand, if $j$ is not an arc at the puncture in $\T$ then it will lie in the area $\calA$ as in Figure~\ref{R} (a).

Part (b) follows from (a) and properties of mutation. Indeed, for $j\neq i$ the entry $2_j$ lies where the corresponding $1_{j}$ will be after mutation at the arc $j$. 
\end{proof}

Since $\calA$ and $\hat \calA$ are identical, and the same occurs with $\calR$ and $\hat \calR$, the statement in the previous lemma can also be reformulated for $\hat{j}$, $\hat \calA$, and $\hat \calR$.

\begin{defi}\label{def:patternG}
We define a {\it pattern $\frakG_{\T}$} using the symmetries of the type $A_{2n-1}$ frieze $\frakF_{\tT}$ as follows.  Take the first $n-1$ nontrivial rows of $\frakF_{\tT}$ together with $1_i$, and remove the region $\calC$ (see Figure  \ref{AR-quiver-1}). The correspondence of $\calR$ and $\hat \calR$ in $\mathfrak{s}$ makes possible to identify $\calR$ and $\hat \calR$ in such way that the integer in the first row in $\hat \calR$, that is $c=r_{n-1}$, is equal to the integer in the $(n-1)$-st row of $\calR$, the same occurs with the second row in $\hat \calR$ and the $(n-2)$-nd row of $\calR$ and so on. Finally, we obtain the desired pattern $\frakG_{\T}$ that contains the glued pattern $g(\frakF_{\T})=\calA\cup \calB$. 
\end{defi}

Observe that the pattern $\frakG_{\T}$  (Figure~\ref{AR-quiver-2}) carries the 1's above the quiddity sequences in $\calA$ and $\hat \calA$ in $\frakF_{\tT}$.   These 1's correspond to boundary segments on the punctured disk $\Sur$, and now they are considered to be part of the regions $\calA$ and $\hat \calA$ as appears in the figure.  For an example of $\frakG_{\T}$ see Figure~\ref{final-example}.

\begin{figure}[h!]
\centering
\scalebox{.87}{\hspace*{-0.6cm}
\def\svgwidth{5.6in}
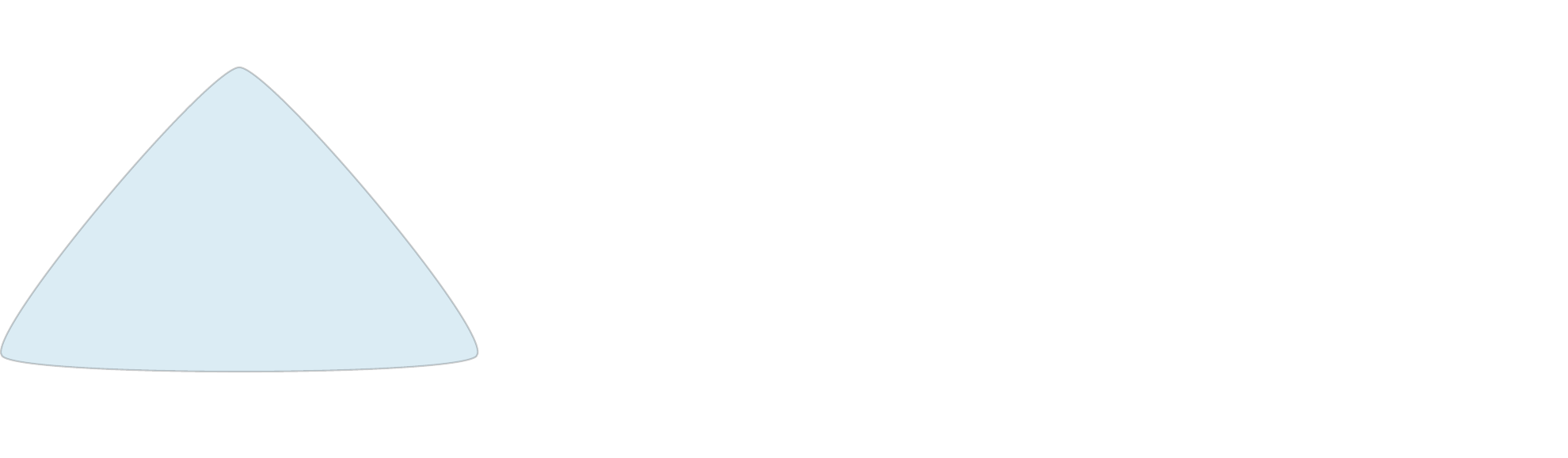}
\caption{Definition \ref{def:patternG} of the Pattern $\frakG_{\T}$.}
\label{AR-quiver-2}
\end{figure}

\begin{remark}\label{properties} $\frakG_{\T}$ has the following properties.

\begin{enumerate}
\item The pattern $g(\frakF_{\T})$ is included in $\frakG_{\T}$ as $\calA \cup \calB$ but also as $\calB \cup \hat \calA$. 
\item By Lemma \ref{lemma-R}, $\calR$ contains the integers associated to arcs at the puncture in $\frakF_{\T}$, so $\frakG_{\T}$ can be easily built from the initial type $D$ pattern $\frakF_{\T}$. 
\item The type $D$ pattern can be recovered from $\frakG_{\T}$. If we denote by $v_1, \ldots, v_{n-1}$ the entries in the longest row of $\calB$ from left to right, then by Definition~\ref{def:gluedfrieze} of gluing we obtain the entries associated to the last two rows of type $D$ frieze. The neighboring entries are of the form $\begin{smallmatrix} v_j / r_j && r_{j+1}\\ r_j && v_{j+1} / r_{j+1}\end{smallmatrix}$.
\item It is locally a type $A$ pattern, i.e. each diamond $\begin{smallmatrix}&a\\b&&c\\&d\end{smallmatrix}$ satisfies $bc-ad=1$.
\end{enumerate}
\end{remark}

\begin{remark}
Recall that we did not define $\mu_a({\bf T})$ when $a=\gamma_r$ is a radius of a self-folded triangle.  However, using tagged triangulations, it is possible to extend the notion of mutation to include this case, see \cite{FST}.  In doing so, the new frieze $\mu_{\gamma_r}(\frakF_{\bf T})$ equals $\mu_{\gamma_l}(\frakF_{\bf T})$ up to interchanging the last two nontrivial rows, where $\gamma_l$ is the loop of this self-folded triangle.  Therefore, without loss of generality we exclude the case when $a$ is a radius of a self-folded triangle when studying mutations of friezes of type $D$. 
\end{remark}

Mimicking the construction in type $A$ (as in Section \ref{section-mutation-type-A} and \cite{baur2018mutation}), we define rays, regions, and projections in $\frakG_{\T}$ induced by mutation at $a$.  Recall that the local configuration of the arc $a$ in $\tilde{\bf{T}}$ is as in Figure~\ref{type-A}, and it defines the arcs $b,c,d,e$ forming a quadrilateral with diagonal $a$.


\subsection{Rays in $\frakG_{\T}$}

In order to describe the action of $\mu_a$ (for $a \neq i$) on a type $D$ frieze, we define special paths and regions over $\frakG_{\T}$ in a similar way that is done in Section \ref{section-mutation-type-A} for a type $A$ frieze.

In this section we  deal with certain sectional paths in the pattern $\frakG_{\T}$.  For that we consider two cases depending on the location of $1_a$ and $2_a$,  either one or none of them are in $\calR$, see Lemma \ref{lemma-A}.

\begin{notation} A sectional path starting at $a$ and ending at $b$ that does not pass through 1's, except maybe for its endpoints, is called a \emph{ray} starting at $a$ and ending at $b$ and denoted by $(a,b)$.
\end{notation}

{\bf Case 1:} We consider $2_{a}, 1_{a} \in \calA$ and $2_a$ is to the left of $1_a$, see Figure~\ref{cases-1}.  We always refer to NE and SE directions in $\frakG_{\T}$ with respect to this figure.

\begin{figure}[h!]
\centering
\scalebox{.92}{\hspace*{-0.6cm}
\def\svgwidth{6.8in}
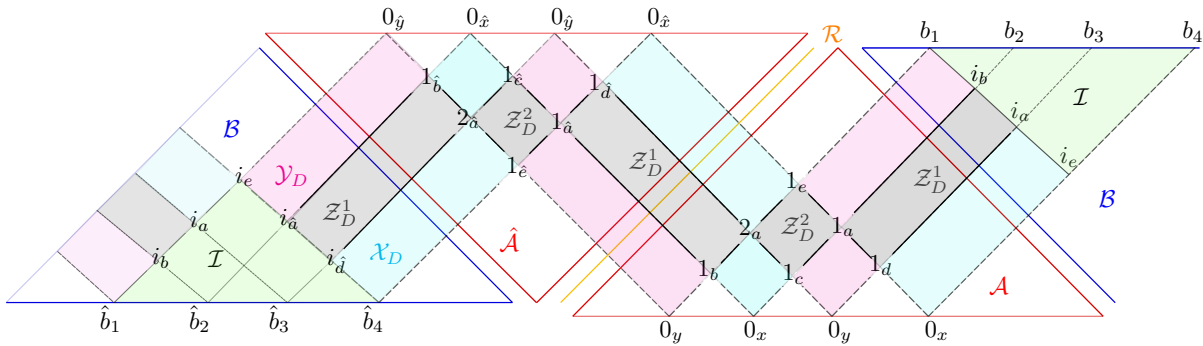}
\caption{Rays and regions in case 1, with $1_{a} \in \calA$ and $2_{a} \in \calA$ is at its left.}  
\label{cases-1}
\end{figure}

Recall the arrows representing elementary moves in the category of arcs, 
introduced after Definition \ref{def:pivot} in Section \ref{section-mutation-type-A}. 
We consider directed paths in $\frakG_{\T}$ using the orientation of these arrows. There is a NE ray in $\frakG_{\T}$  given by $(1_{c}, 1_{a})$, and a SE ray $(1_{e},1_{a} )$. There is a NE ray $(2_{a}, 1_{e})$ and a SE ray $(2_{a},1_{c})$ as the reader can see in Figure \ref{cases-1}. By Section \ref{section-mutation-type-A}, we know that the rays described above bound a rectangular region $\calZ$ in $\frakF_{\tT}$, and moreover there is a NE ray $(1_{b}, 2_{a})$ and a SE ray $(1_{a}, 1_{d})$ that bound another rectangular region $\calZ$ in $\frakF_{\tT}$.

The rays mentioned above can be extended downwards until each one reaches a 0, denoted by $0_{x}$ or $0_{y}$, in the trivial row.  For each one of these four entries $0_{y}, 0_{x}, 0_{y}, 0_{x}$ in the inferior row, there are two maximal sectional paths (from now on \emph{ m.s.p.} to abbreviate) one starting and the other one ending at the said entry.  We denote by $b_1,b_2, b_3, b_4$, from left to right, the last element of the m.s.p starting at the zeroes. Similarly, we denote by $0_{\hat y}$ and $0_{\hat x}$ the first elements of the  m.s.p. ending at the $0_x$'s and $0_y$'s respectively.

Now we repeat the process. There are four m.s.p. ending at those $0_{\hat y}$ and $0_{\hat x}$, and we denote by $\hat b_1,\hat b_2, \hat b_3, \hat b_4$, from left to right, the first element of the m.s.p.'s starting at the zeroes. We will see that all the intersections of the twelve m.s.p's above determine starting and ending points of rays $(1_{\hat c}, 1_{\hat a}), (1_{\hat e}, 1_{\hat a})$ and rays $(2_{\hat a}, 1_{\hat e}), (2_{\hat a}, 1_{\hat c})$ as we see in the Figure \ref{cases-1}. Moreover we will prove in this section that $b_i = \hat b_i$, therefore the rays mentioned bound rectangular regions in the pattern.

First, we need the following result.  

\begin{lemma}\label{lem:ray}
In Case 1 there exists a ray $(1_{\hat d}, 2_a)$ in the pattern $\frakG_{\T}$. 
\end{lemma}

\begin{proof}
Consider the surface $(\tilde{\bf S}_i, {\tT})$ and the local configuration of the arc $a$ depicted in Figure~\ref{gamma}(a).  Note that some of the arcs $b,c,d$ and their respective counterparts $\hat b, \hat c, \hat d$ may be boundary segments of $\tilde{\bf S}_i$.  By Remark~\ref{frieze_entries}(2), the entry $2_a$ in the frieze $\frakF_{\tT}$ corresponds to the arc $a'$ that crosses only one arc $a$ in the triangulation $\tT$.  The endpoints of $a'$ are $t_1, t_2$, where $t_1$ is the common endpoint of $c$ and $d$, while $t_2$ is the common endpoint of $e$ and $b$.  
 
Now consider the entry $2_a$ in the pattern $\frakG_{\T}$.  There are two maximal sectional paths in $\frakG_{\T}$ ending in $2_a$.  One of them contains the NE ray $(1_b, 2_a)$ as in Figure~\ref{cases-1}.  Let $(r,2_a)$ be the second sectional path in $\frakG_{\T}$ ending in $2_a$ and starting in some $r\in \calR$.   By Remark~\ref{sec_paths}, this sectional path corresponds to a sequence of clockwise elementary moves starting with the arc $a'$ and moving its endpoint $t_2$ clockwise until we obtain an arc $\gamma_r$ with endpoints $q$ and $t_1$.   This arc corresponds precisely to the entry $r$ in the frieze by Lemma~\ref{lemma-R}.  The same lemma also implies that the associated entry $\hat r\in \hat{\calR}$ in the frieze $\frakF_{\tT}$ corresponds to the arc $\gamma_{\hat{r}}$ with endpoints $q$ and $\hat{t}_1$.   Now starting with $\gamma_{\hat r}$ we can again perform moves that move $q$ clockwise while keeping $\hat{t}_1$ fixed until we obtain the arc $\hat d$.   This implies that in $\frakF_{\tT}$ there is a sectional path $(1_{\hat d}, \hat{r})$. By construction of $\frakG_{\T}$ in Definition~\ref{def:patternG} the two sectional paths $(1_{\hat d}, \hat{r}), (r, 2_a)$ in $\frakF_{\tT}$   become a single sectional path $(1_{\hat d}, 2_a)$ in $\frakG_{\T}$ induced by identifying $\calR$ and $\hat{\calR}$.   Finally, we observe that this path is actually a ray, because the sequences of elementary moves described above always yield arcs that cross elements of ${\tT}$.  This implies that the path $(1_{\hat d}, 2_a)$ does not pass through any 1's as desired.  
\end{proof}

Let $\gamma_{\hat d}, \gamma_{b}$ be two maximal sectional paths in $\frakG_{\T}$ ending in $1_{\hat d}, 1_{b}$ respectively, such that the compositions $\gamma_{\hat d}\circ(1_{\hat d}, 2_a),  \gamma_{b}\circ (1_b, 2_a)$ are not sectional. With this notation consider the following result.

\begin{lemma}\label{gluerays}
The paths $\gamma_{\hat d}, \gamma_{b}$ in $\frakG_{\T}$ intersect in entry $1_{\hat a}$.
\end{lemma}
\begin{proof} 
Recall that $\calA$ and $\hat \calA$ are identical sub-patterns of $\frakF_{\tT}$, hence $1_{a}, 1_{d}$ and $1_{\hat a},1_{\hat d}$ lie in the same positions in $\calA$ and $\hat \calA$. The same holds for $1_b,2_a$ and $1_{\hat b}, 2_{\hat a}$. This implies that  the rays $(1_a, 1_b)$ and $(1_d,2_a)$, considered in $\frakF_{\tT}$, cross $\calR$ at $r_i$ and $r_j$ respectively, with $i < j$, while the rays $(1_{\hat a}, 1_{\hat b})$ and $(1_{\hat d}, 2_{\hat a})$ cross $\hat{\calR}$ at $\hat r_i$ and $\hat r_j$. We will use the identification $r_i \leftrightarrow \hat r_i$ given in Lemma \ref{lemma-R}.
From Definition~\ref{def:patternG} of $\frakG_{\T}$, we obtain $\gamma_{b}$ by gluing two paths $( - , \hat r_i)$ and $(r_i,1_b) $ in the type $A$ frieze. Observe that $(r_i, 1_b)$ is in fact a ray and $r_i \neq 1$ because in $\frakF_{\tT}$ it is a part of the ray $(1_a, 1_b)$. Similarly, the first 1 in NW direction from $\hat r_i$ has to be $1_{\hat a}$, that is to say $(1_{\hat a}, \hat r_i)$ is a ray because in $\frakF_{\tT}$ it is a part of the ray $(1_{\hat{a}}, 1_{\hat{b}})$.  This shows that in $\frakG_{\T}$ the first 1 in NW direction from $1_b$ is $1_{\hat a}$. On the other hand, since $\calA$ and $\hat \calA$ are identical, we have a ray $(1_a, 1_d)$ and a corresponding one $(1_{\hat a}, 1_{\hat d})$, so  in $\hat{\calA}$ the first 1 in SW direction from $1_{\hat d}$ is $1_{\hat a}$ as desired.   
\end{proof}

This shows that we have rays $(0_{\hat y}, 1_{\hat b}), (1_{\hat b}, 0_{\hat x}), (0_{\hat x}, 1_{\hat c}), (1_{\hat c}, 0_{\hat y}), (0_{\hat y}, 1_{\hat d}), (1_{\hat d}, 0_{\hat x})$,
together with the corresponding ones in $\calA$. We also have rays starting at $1_{\hat e}$, $1_{\hat a} $, $1_{\hat d}$, $0_{\hat x}$ crossing $\calR$ and ending at $0_y$, $1_b$, $2_a$, $1_e$ respectively. To show that the rays $(1_{\hat e}, 0_y)$ and $(0_{{\hat x}}, 1_e)$  pass through $\calR$ as in the figure,  the same argument as in the proof of Lemma~\ref{gluerays} can be used, and therefore we omit it.

On the other hand, if $1_a, 2_a \in \calA$ and $2_a$ appears to the right of $1_a$, then we still obtain the exact same configuration of rays as in Figure~\ref{cases-1}, by replacing $1_a, 1_{\hat a}$ with $2_a, 2_{\hat a}$ respectively.  Note that instead of $\frakG_{\T}$ this relabeling corresponds to looking at the pattern $\frakG_{\mu_a(\T)}$ and the associated paths determined by mutation at $a'$, see Definition~\ref{def:mut}.  This completes the description of rays in case 1.

{\bf Case 2:} The remaining possibility, is whenever $1_a, 2_a$ are not both in $\calA$.  The case when $1_a\in \calR$ and $2_a\in \calA$ is depicted in Figure~\ref{cases-2}.  

\begin{figure}[h!]
\centering
\scalebox{.92}{\hspace*{-0.6cm}
\def\svgwidth{6.6in}
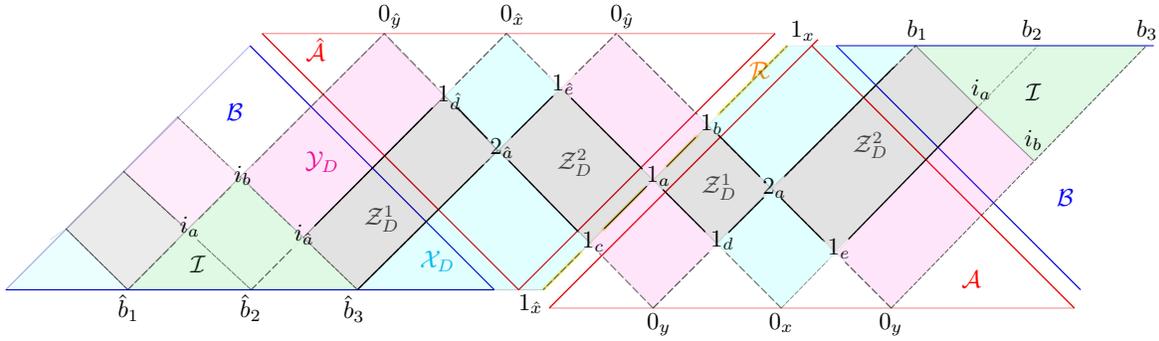}
\caption{Rays and regions in case 2, with $1_a \in \calR$.}
\label{cases-2}
\end{figure}

The definition of the corresponding rays in $\frakG_{\T}$ is similar to the previous case, so we omit the detailed discussion.  Thus, for the description of these paths we refer to the figure.  On the other hand, if $1_a\in \calA$ and $2_a\in \calR$  then again we obtain the exact same configuration as in the case $1_a\in \calR$ and $2_a\in \calA$ by interchanging the position of $1_a, 1_{\hat a}$ with $2_a, 2_{\hat a}$.  This completes the description of rays in all cases.

\smallskip

Next we find a precise relationship between paths starting and ending in the longest row of $\calB$.  In particular, in the case $1_a, 2_a \in \calA$ there are four rays $(\hat b_1, 0_{\hat y}), (\hat b_2, 1_{\hat b}), (\hat b_3, 2_{\hat a}), (\hat b_4, 1_{\hat e})$ starting in the last row of $\calB$ and four rays $(1_e, b_1), (1_a, b_2), (1_d, b_3), (0_{{x}}, b_4)$ ending in the last row of $\calB$.  We claim that these two sets of rays pairwise intersect in the last row of $\calB$.  Analogous situation occurs in the case when $1_a$ or $2_a$ belongs to $\calR$.
Thus, we can think of the former set of rays being obtained from the latter by ``bouncing off" the boundary of $\calB$.

\begin{lemma}\label{overlap}
In $\frakG_{\T}$ the entries $b_j=\hat b_j$ for all $j=1, \dots, 4$. 
\end{lemma}

\begin{proof}
We prove this claim for a particular pair of rays, however the same argument can be used to justify the rest.  Suppose $1_a,2_a\in \calA$ as in Figure~\ref{cases-1}.  Then we want to show that the rays $(1_a, b_2)$ and $(\hat b_2, 1_{\hat b})$ intersect in the last row of $\calB$.
The local configuration of the arc $a$ in the polygon $\tilde{\bf S}_i$ is depicted in Figure~\ref{gamma}.  Note that some of the arcs $b,c,d$ and their respective counterparts $\hat b, \hat c, \hat d$ may be boundary segments of $\tilde{\bf S}_i$.  Let $\delta$ denote the arc in $\tilde{\bf S}_i$ corresponding to $a_1$, for $a_1$ as in Definition~\ref{def:bcd}.  Also, let $\gamma$ be the arc in $\tilde{\bf S}_i$ with endpoints $s$ and $\hat{s}$ as in Figure~\ref{gamma}.  First, we claim that its corresponding entry $m_{\gamma}$ in the frieze belongs to the last row of $\calB$.

\begin{figure}[h!]
\centering
\scalebox{.82}{\hspace*{-0.6cm}
\def\svgwidth{4.5in}
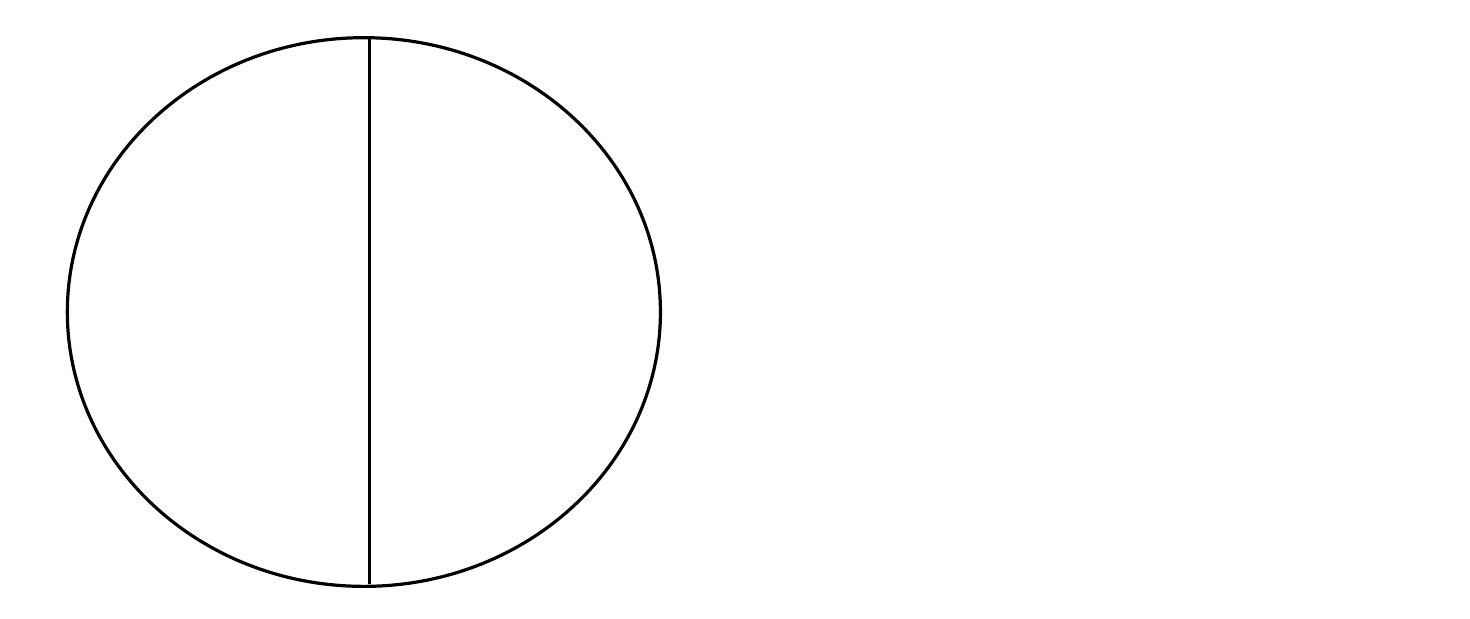}
\caption{(a) Lemma \ref{lem:ray} and \ref{overlap}; (b) Lemma \ref{lastcase}.}
\label{gamma}
\end{figure}

Recall Definition~\ref{def:pivot} of elementary moves in the category of arcs.  Note, that the construction of $\tilde{\bf S}_i$ implies that if it takes $k$ elementary moves to go from $s$ to $\delta$'s nearest endpoint, then it takes $n-2 - k$ moves to go from $\hat{s}$ to $\delta$'s other endpoint.  Moreover, the longest row of $\calB$ is characterized by this; it is the row that can be reached from $a_1$ after $n-2$ elementary moves always moving down from one row to the next (if $a_1$ is considered as the bottom row of $\calB$), see Figure~\ref{AR-quiver-1}.  Therefore, this shows the claim that $m_{\gamma}$ belongs to the the longest row of $\calB$.

Now, given $\gamma$ and moving $\hat{s}$ clockwise while keeping $s$ fixed we obtain the arc $a$.  This shows that there exists a sectional path in $\frakG_{\T}$ starting in $1_a$ and ending in $m_{\gamma}$,  see Remark~\ref{sec_paths}.  This implies that $m_{\gamma}$ is the endpoint of $(1_a, b_2)$ in the longest row of $\calB$.  Similarly, moving $s$ counterclockwise while keeping $\hat{s}$ fixed we obtain the arc $\hat{b}$.  Thus, there exists a sectional path in $\frakG_{\T}$ starting in $m_{\gamma}$ and ending in $1_{\hat b}$.  This proves that the starting point of $(\hat{b}_2, 1_{\hat{b}})$ and the endpoint of $(1_a, b_2)$ in $\calB$ coincide. 
\end{proof}

In view of the last lemma, from now on we will refer to both $b_j$ and $\hat b_j$ as $b_j$.


\subsection{Regions in $\frakG_{\T}$}

Here we define regions in $\frakG_{\T}$ that are determined by the set of rays bounding them. Except for the region $\mathcal{I}$ that arises as the intersection of certain regions in $\frakF_{\tT}$ determined by $\mu_a$ and $\mu_{\hat a}$, the remaining regions in $\frakG_{\T}$ are analogous to the corresponding ones in type $A$, see Definition \ref{def_regionsA}.

{\bf Region $\mathcal{I}$:} 
It is defined as the triangular region in $\calB$ bounded by $(i_e, b_4)$ and $(b_1, i_e)$ in Figure \ref{cases-1}, and by
$(b_1, i_b)$ and $(i_b, b_3)$ in Figure \ref{cases-2}.


{\bf Region $\calZ_D$:} 
It is given by the entries in $\frakG_{\T} \setminus \mathcal{I}$ that are bounded by rays starting or ending in $1_a, 1_{\hat a}$ and $2_a, 2_{\hat a}$. For example, in Figure~\ref{cases-1}, $\mathcal{Z}_D$ is given by five rectangles with vertices 
$$\{i_{\hat a}, 1_{\hat b}, 2_{\hat a}, i_{\hat d} \} \hspace{.5cm} \{ 2_{\hat a}, 1_{\hat c}, 1_{\hat a}, 1_{\hat e} \} \hspace{.5cm} \{1_{\hat a}, 1_{b}, 2_{a}, 1_{\hat d} \} \hspace{.5cm} \{ 2_a, 1_c, 1_a, 1_e \} \hspace{.5cm} \{ 1_a, 1_d, i_a, i_b  \}, $$

In Figure~\ref{cases-2}, $\mathcal{Z}_D$ consists of four rectangular regions. Note that some of these rectangles might be empty, which occurs when two  parallel sides of the rectangle are neighboring rays. 

We also let $\overline{\mathcal{Z}}_D$ denote all entries in $\frakG_{\T}$ that belong to $\mathcal{Z}_D$ or the rays that bound $\mathcal{Z}_D$.  Observe that $\calZ_D = \calZ^1_{D} \cup \calZ^2_{D}$, where ${\calZ}^1_{D}$ is the union of rectangular regions that have a 1 as a leftmost corner or a 2 as a rightmost corner, and ${\calZ}^2_{D}$ is given by rectangles that have a 2 as a leftmost corner or a 1 as a rightmost corner.

{\bf Regions $\calY_D$ and $\calX_D$:}
In Figure \ref{cases-1}, let $\mathcal{Y}_D$ denote the entries in one of the five rectangular regions with vertices $$\{ 1_e,1_a,b_1, i_b \} \hspace{.5cm} \{i_e, i_{\hat a},0_{\hat y}, 1_{\hat b} \} \hspace{.5cm} \{ 0_{y}, 1_{d}, 1_{a}, 1_{c} \}\hspace{.5cm} \{ 0_{\hat y}, 1_{\hat d}, 1_{\hat a}, 1_{\hat c} \} \hspace{.5cm} \{ 0_{y}, 1_{b}, 1_{\hat a}, 1_{\hat e} \},  $$ 
and let $\mathcal{X}_D$ denote the entries in one of the five rectangular regions with vertices
$$\{ i_{\hat d},b_{4},2_{\hat a},1_{\hat e} \}\hspace{.5cm} \{1_{\hat b}, 0_{\hat x}, 1_{\hat c}, 2_{\hat a} \} \hspace{.5cm} \{ 1_{\hat d}, 0_{\hat x}, 1_{e}, 2_a \} \hspace{.5cm} \{ 1_b, 2_a, 1_c,  0_x \} \hspace{.5cm} \{ 1_d, 0_x, i_e, i_a \}. $$ 
We have corresponding regions in Figure \ref{cases-2}. Moreover, we include in $\calY_D$ (resp. $\calX_D$) the interior of the rays that bound both $\calY_D$ (resp. $\calX_D$) and $\mathcal{I}$. For example, in Figure \ref{cases-1}, the interior of $(i_a, i_e)$ belongs to $\calX_D$. As happens with $\calZ_D$, some rectangles in $\calX_D$ or $\calY_D$ can be empty. Notice that $\calX_D, \calY_D$ interchange positions if $1_a$ and $2_a$ switch, that is, if we look at $\frakG_{\mu_a(\T)}$ instead of $\frakG_{\T}$.

{\bf Region $\calF_D$:}
Finally, we define $\mathcal{F}_D$ to be the set of entries in $\frakG_{\T}$ that are not in $\overline{\mathcal{Z}}_D\cup\mathcal{X}_D\cup\mathcal{Y}_D \cup \mathcal{I}$. 

\smallskip

Recall that in $\frakF_{\tT}$ we also have regions $\mathcal{Z}, \overline{\mathcal{Z}}, \mathcal{X}, \mathcal{Y}, \mathcal{F}$ determined by mutation at $a$ and also regions $\hat{\mathcal{Z}}, \overline{\hat{\mathcal{Z}}}, \hat{\mathcal{X}}, \hat{\mathcal{Y}}, \hat{\mathcal{F}}$ determined by mutation at $\hat a$.   Given a region $\mathcal{W}\in \frakF_{\tT}$ let $\mathcal{W}_{\frakG}$ denote its restriction to the pattern $\frakG_{\T}$.

\begin{remark} 
For every region $\mathcal{W} \in \{\mathcal{X} , \mathcal{Y} , \mathcal{Z}, \mathcal{F} \}$ in the frieze $\frakF_{\tT}$ there is a corresponding region $\mathcal{W}_D$ in the pattern $\frakG_{\T}$ such that $\mathcal{W}_D =(\mathcal{W}\cup\hat{\mathcal{W}}\setminus \mathcal{I})_{\frakG}$ for $\mathcal{W} \in   \{\mathcal{X} , \mathcal{Y} , \mathcal{Z}\}$ and $\mathcal{F}_D =\mathcal{F}\cap \hat{\mathcal{F}}$. This follows from the definition of regions $\mathcal{W}_D$ above and $\mathcal{W}$ (Definition~\ref{def_regionsA}). Note, that in the type $A$ frieze the regions $\mathcal{X}, \mathcal{Y}, \mathcal{Z}$ determined by $\mu_a$ are disjoint, as it is always the case for a single mutation in type $A$. However, if we consider all the regions determined by the pair of mutations $\mu_a, \mu_{\hat a}$ in $\frakF_{\tT}$ then a couple of them may overlap. The region $\mathcal{I}$ corresponds to these intersections.   
\end{remark}


\subsection{Projections in $\frakG_{\T}$}

Here we define various projections for entries in $\overline{\calZ}_D, \calX_D, \calY_D$, and projections onto $\calA,\calR$ for entries in $\mathcal{I}$.  It is important to note that the former ones are similar to projections in type $A$ from Definition~\ref{projections-A} shown in Figure~\ref{fig:regions}. The only difference occurs when one of the sectional paths involved in the definition of projections in $\frakG_{\T}$ passes through $\mathcal{I}$.  In this case we introduce additional paths as follow.

\begin{defi}\label{gamma-msp}
Let $\gamma$ be a maximal sectional path in $\frakG_{\T}$ starting (resp. ending) in the last row of $\calB$. Then the other endpoint of $\gamma$ is $0_{\gamma}$ in the row of 0's in $\hat \calA$ (resp. $\calA$).
Define ${\overrightarrow{\gamma}}$ (resp. ${\overleftarrow{\gamma}}$) to be the maximal sectional path starting (resp. ending) in $0_{{\gamma}}$.  See Figure~\ref{def-410}.
\end{defi}

\begin{figure}[h!]
\centering
\scalebox{.9}{
\def\svgwidth{4.8in}
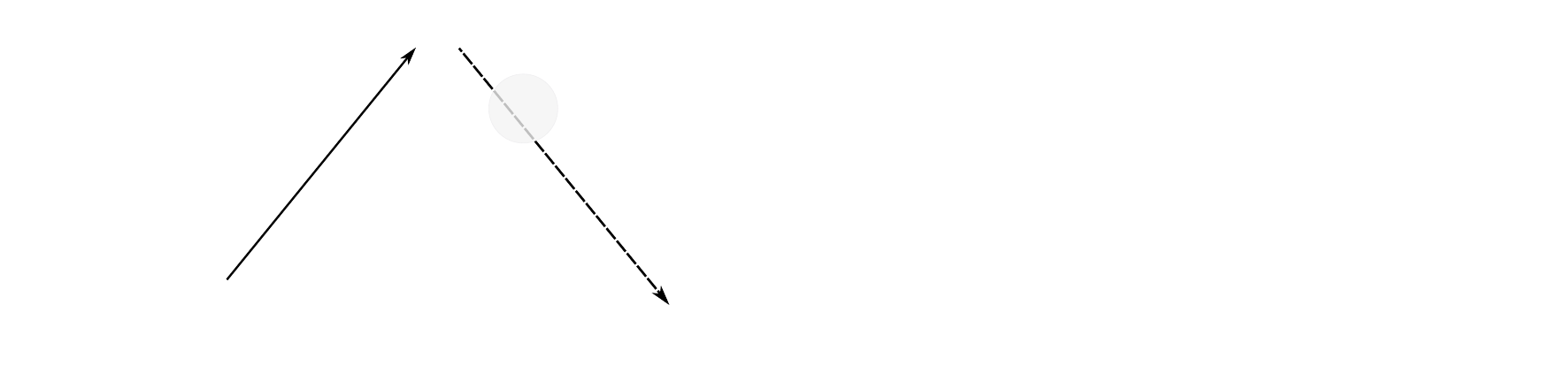}
\caption{Definition \ref{gamma-msp}.}
\label{def-410}
\end{figure}

\begin{defi}
Let $\partial \calZ$ be the set of rays bounding $\calZ_D$ that do not bound $\mathcal{I}$. As before, $\partial \calZ$ is a union $\partial \calZ_1 \cup  \partial \calZ_2$, where a ray is in $\partial \calZ_t$ if it bounds $\calZ^t_D$ for $t=1,2$.
\end{defi}

Given $m \in \overline{\mathcal{Z}}_D$, we determine four projections $\rho^i_j(m)$, where $i \in \{ \uparrow, \downarrow \}$, $j\in \{ p,s\}$. Let  $\gamma_1$ be the m.s.p. passing through $m$ in SE direction, and $\gamma_2$ be the m.s.p. passing through $m$ in NE direction.   For the next set of definitions we refer to Figure~\ref{fig:projXYZ}.

\begin{defi} (Projections of entries in $\overline{\calZ}_D$) Let $m \in \overline{\calZ}^1_D$ (resp. $\overline{\calZ}^2_D$). The projection $\rho^{\uparrow}_p (m)$ (resp. $\rho^{\uparrow}_s (m)$) is the first entry of $\partial \calZ$ crossed by $\gamma_1$, and the projection $\rho^{\downarrow}_s (m)$ (resp. $\rho^{\downarrow}_p (m)$) is the second entry of $\partial \calZ$ crossed by $\gamma_1$. Analogously, the projection $\rho^{\downarrow}_p (m)$ (resp. $\rho^{\downarrow}_s (m)$) is the first entry of $\partial \calZ$ crossed by $\gamma_2$, and the projection $\rho^{\uparrow}_s (m)$ (resp. $\rho^{\uparrow}_p (m)$) is the second entry of $\partial \calZ$ crossed by $\gamma_2$.
\end{defi}

Depending on the position of $m$, one of these projections might not be defined. This happens precisely when $\gamma_2$ crosses $\mathcal{I}$.  In this case, we obtain the last projection as follows.  Here, let $\overline{\mathcal{I}}$ denote entries in $\mathcal{I}$ and the rays bounding it.

\begin{defi} (Projections of entries in $\overline{\calZ}_D$ continued)
Let $m \in \overline{\calZ}^1_D$ (resp. $ \overline{\calZ}^2_D$). If the starting point of $\gamma_2$ is in $\overline{\mathcal{I}}$, then let the remaining projection $\rho_p^{\downarrow} (m)$ (resp. $\rho_s^{\downarrow} (m)$) be the first intersection of $\overrightarrow{\gamma_2}$ and $\partial \calZ_1$ (resp. $\partial \calZ_2$).  If the ending point of $\gamma_2$ is in $\overline{\mathcal{I}}$, then let the remaining projection $\rho_s^{\uparrow} (m)$ (resp. $\rho_p^{\uparrow} (m)$) be the last intersection of $\overleftarrow{\gamma_2}$ and $\partial \calZ_1$ (resp. $\partial \calZ_2$).
\end{defi}

\begin{figure}[h!]
\centering
\scalebox{.88}{\def\svgwidth{7.1in}
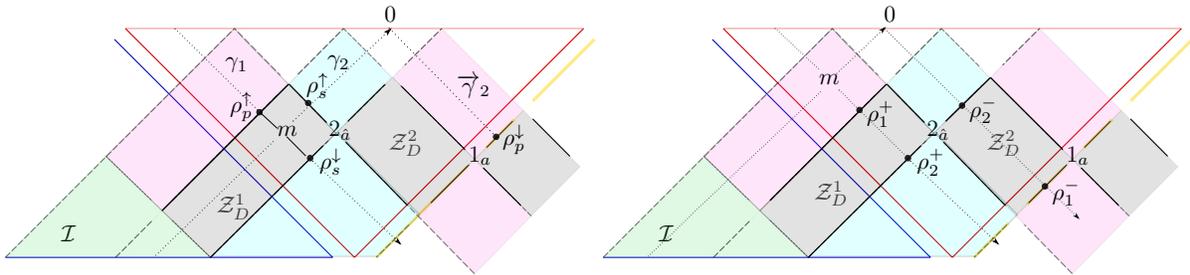}
\caption{Definition of projections in $\calZ_D$ (left) and $\calY_D$ (right) when one of the sectional paths starts at $\mathcal{I}$.}
\label{fig:projXYZ}
\end{figure}

\begin{defi}(Projections of entries in ${\mathcal{X}}_D$ and ${\mathcal{Y}}_D$)
Given $m$ in $\calX_D$ or $\calY_D$ there exist unique sectional paths  $\gamma^+$ and $\gamma^-$  starting and ending at $m$ respectively, such that $\gamma^+$ and $\gamma^-$ intersect $\partial \calZ \cup \mathcal{I}$, and $\gamma^- \circ \gamma^+$ is not sectional. The projection of $m$ onto the closest ray of $\partial \calZ$ along $\gamma^+$ is denoted by $\rho^+_1(m)$ and the projection onto the second closest ray is $ \rho^+_2(m)$. The projection of $m$ onto the closest ray of $\partial \calZ$ along $\gamma^-$ is denoted by $\rho^-_1(m)$ and onto the second closest ray is denoted by $\rho^-_2(m)$.
\end{defi}

Note that some of these projections might not exist when $\gamma^+$ or $\gamma^-$ end or start in $\calB$ prior to intersecting one or both of the rays in $\partial \calZ$.  In this case, let $\gamma_+$ or $\gamma_-$ be a maximal sectional path in $\frakG_{\T}$ obtained from extending $\gamma^+$ or $\gamma^-$ respectively.

\begin{defi}(Projections of entries in ${\mathcal{X}}_D$ and ${\mathcal{Y}}_D$ continued) Define the remaining projections of $m$ to be the respective projections of $\overleftarrow{\gamma_+}$ or $\overrightarrow{\gamma_-}$ onto rays in $\partial \calZ$. More precisely, if both projections along $\gamma^+$ do not exist then define $\rho^+_1(m)$ and $\rho^+_2(m)$ to be the second to last and last intersection along $\overleftarrow{\gamma_+}$ with these rays respectively.  Similarly, if both projections along $\gamma^-$ do not exist then define $\rho^-_1(m)$ and $\rho^-_2(m)$ to be the second and first intersection along $\overrightarrow{\gamma_-}$ with these rays respectively. If just one of the projections along $\gamma^{+}$ or $\gamma^-$ was not defined, the remaining projection is obtained via $\overleftarrow{\gamma_+}$ or $\overrightarrow{\gamma_-}$ as the last or first intersection with $\partial\calZ$.
\end{defi}

The following lemma relates projections in $\frakF_{\tT}$ and $\frakG_{\T}$.  Given an entry $m\in \frakG_{\T}$ we can lift it to an entry $\tilde{m}\in \frakF_{\tT}$.   This lift is unique, except for the elements of $\calR$.  In this case there are two copies one in $\calR$ and the other one in $\hat \calR$, 
however the projections in $\frakF_{\tT}$ of the two lifts onto rays determined by mutation at $\mu_a$ and $\mu_{\hat a}$ respectively will be the same due to the particular symmetry of this type $A$ frieze.  Therefore, we may choose either of the lifts and the following notation makes sense.

For each projection $\rho(m)$ in $\frakG_{\T}$ there is a corresponding projection $\pi(\tilde{m})$ in $\frakF_{\tT}$ of its lift.  
  
\begin{lemma}\label{projlemma}
Let $m$ be an entry in $\frakG_{\T}$ that belongs to a region $\mathcal{W}_D$. Then each projection $\rho(m)$ in $\frakG_{\T}$ equals its corresponding projection $\pi(\tilde{m})$ in $\frakF_{\tT}$.
\end{lemma}

\begin{proof}  
Most projections in $\frakG_{\T}$ are given by sectional paths identical to the ones appearing in $\frakF_{\tT}$. However, we can distinguish two cases arising in a different way. 

The first case occurs when the sectional path defining the projections crosses $\calR$. Let $m \in \mathcal{W}_D \subset \hat\calA \cup \calR \cup \calA$ and let $(m_1,m_2)$ be the path used to compute $\rho(m)$ in $\frakG_{\T}$, see Figure \ref{AR-quiver-2}. This path is the projection of two paths from $\frakF_{\tT}$, namely $(\tilde{m}_1, \tilde{r}_i)$ that starts at $\hat\calA$ and ends at $\hat \calR$ (see the NW corner of $\frakF_{\tT}$ in Figure \ref{AR-quiver-1}), and $(\tilde{r}_i, \tilde{m}_2)$ that starts at $\calR$ and ends at $\calA$ (see the SW corner of $\frakF_{\tT}$ in Figure \ref{AR-quiver-1}). For $\tilde{m} \in (\tilde{m}_1, \tilde{r}_i)$ the projection $\pi(\tilde{m})$ in  $\frakF_{\tT}$  that is obtained by crossing $\hat{\calR}$ belongs to $\hat \calA$, since there are no $1$'s nor $2_{a}, 2_{\hat a}$ in $\calC \cup \calB$. We know by Lemma \ref{gluerays} that the rays starting or ending at $1_a,2_a \in \calA$ and the corresponding ones starting or ending at $1_{\hat a}, 2_{\hat a} \in \hat \calA$ look identical in both copies. Hence $\pi(\tilde{m}) = \rho(m)$. The same occurs for $\tilde{m} \in (\tilde{r}_i, \tilde{m}_2)$.  

The second case is when we use an auxiliary maximal path $\overrightarrow{\gamma}$ or $\overleftarrow{\gamma}$. We will explain the case of $\overrightarrow{\gamma}$. Let $m$ be an entry in $\calB \cup \hat \calA$, see Figure \ref{AR-quiver-2}. Suppose that $m$ lies in a maximal path $\gamma=(x, 0_\gamma)$, where $x \in \mathcal{I}$, and we need to use $\overrightarrow{\gamma}$ to compute a single projection, if $m \in \overline{\calZ}_D$, or at most two projections, if $m $ lies in $\calX_D$ or $\calY_D$. The associated projection $\pi(\tilde{m})$ in $\frakF_{\tT}$ is obtained as an intersection of the ray $\tilde{\gamma}$, that passes through $\tilde{x}$ and ends at $\tilde{m}$, and rays starting or ending at $1_{\hat a}, 2_{\hat a}$.  For every frieze of type $A$, and particularly in our setting, the NE maximal path $(0, 0_\gamma)$ in $\frakF_{\tT}$ is equal to the SE maximal path $(0_\gamma, 0)$. Thus, the entries obtained as an intersection of $\tilde{\gamma}$ and the boundary of $\hat \calZ$ will coincide with the ones we find using $\overrightarrow{\gamma}$. 
Observe that in the case of two projections we have that first $\pi^{-}_1(\tilde{m})$ and second $\pi^{-}_2(\tilde{m})$ projections along $\tilde{\gamma}$ correspond to second and first projections $\rho_1^-(m)$ and $\rho_2^-(m)$ along  $\overrightarrow{\gamma}$  respectively.  
\end{proof}

Next, we consider projections defined for entries in $\calB$, and particularly for those entries that belong to $\mathcal{I}$.

\begin{defi}\label{defi:projR}(Projections onto $\mathcal{R}$)
Given an entry $m$ lying in $\mathcal{B}$, let $\gamma_+$ (resp. $\gamma_-$) be a maximal sectional path in $\frakG_{\T}$ passing through $m$ and ending (resp. starting) in the last row of $\calB$.  Then  $\overleftarrow{\gamma_+}$ (resp. $\overrightarrow{\gamma_-}$) passes through $\calR$, define $\rho_{R^+}(m)$ (resp. $\rho_{R^-}(m)$) to be the unique entry in $\overleftarrow{\gamma_+}\cap \calR$ (resp. $\overrightarrow{\gamma_-}\cap \calR$). 
\end{defi}

\begin{defi}(Projections onto $\mathcal{A}$ and $\hat{\mathcal{A}}$)
Given an entry $m$ lying in $\mathcal{B}$, there are two rays $\gamma^-, \gamma^+$ in $\frakG_{\T}$ passing through $m$ and starting/ending in $\mathcal{A}, \hat{\mathcal{A}}$ respectively.  Let $\rho^-_A(m)$ denote the entry in $\mathcal{A}$ lying on the ray $\gamma^-$ such that all proper successors of $\rho^-_A(m)$ on $\gamma^-$ belong to $\mathcal{B}$.   Similarly, let $\rho^+_A(m)$ denote the entry in $\hat{\mathcal{A}}$ lying on the ray $\gamma^+$ such that all proper predecessors of $\rho^+_A(m)$ on $\gamma^+$ belong to $\mathcal{B}$.
\end{defi}

\begin{lemma}\label{lastcase}
For any $m\in \mathcal{B}$ 
$$m=\rho_R^+(m)\rho_A^+(m)+\rho_R^-(m)\rho_A^-(m).$$
\end{lemma}

\begin{proof}
Let $m_{\gamma}$ be an entry in $\calB$ corresponding to an arc $\gamma\in \tilde{\bf S}_i$ in the polygon, see Figure~\ref{gamma}(b).  By Ptolemy relations we obtain 
$$m_{\gamma}m_{i} = m_{\hat \alpha}m_{\beta}+m_{\alpha}m_{\hat \beta}$$
and note that $m_{i}=1$ since the arc $i$ is a part of the triangulation, see Remark~\ref{frieze_entries}(1).  Next, we show that this relation is the same as the one appearing in the statement of the lemma. 

First, observe that $m_{\hat \alpha}\in \hat \calA$ and $m_{\alpha}\in\calA$.   Moreover, there exist unique sectional paths $m_{\gamma} \to \cdots \to m_{\hat \alpha}$ and $m_{\alpha} \to \cdots \to m_{\gamma}$ in $\frakG_{\T}$.  These paths lie in $\calB$ except for the endpoints $m_{\alpha_1}, m_{\alpha}$.  This implies $m_{\alpha}=\rho^-_A(m_{\gamma})$ and $m_{\hat \alpha}=\rho^+_A(m_{\gamma})$.

Similarly, $m_{\beta}\in \calR$ and $m_{\hat \beta}\in \hat \calR$ in  $\frakF_{\tT}$. Moreover, $\gamma$ is obtained from $\beta$ in a sequence of elementary moves that correspond to desired sectional paths in $\frakG_{\T}$ as follows.  First, move $\hat{s}$ clockwise until it reaches $s$ while keeping $s$ fixed.  Thus, obtain a point $s$ on the boundary, which corresponds to some entry $0_{\beta}$ in the row of zeros in $\calA$.  We can think of $s$ as an arc in $\tilde{\bf S}_i$ where both of its endpoints are $s$.  Next, starting from this arc at $s$ obtain $\beta$ by moving one of its endpoints clockwise until it reaches $q$ while keeping the other fixed.  Observe that this sequence of moves in $\tilde{\bf S}_i$ from $\gamma$ to $\beta$ yields a sequence of sectional paths in $\frakG_{\T}$ in Definition~\ref{defi:projR} of $\rho_R^-(m_{\gamma})$.  This implies that $m_{\beta}=\rho_R^-(m_{\gamma})$.  Similarly, we conclude that $m_{\hat\beta}=\rho_R^+(m_{\gamma})$.  
\end{proof}

\begin{theorem}\label{main-theorem}
Let $m\in \frakG_{\T}$, and let $m'\in \frakG_{\T'}$ be the corresponding integer after mutation at vertex $a$.  Then $\delta_a (m):= m-m'$ is computed as follows.

\begin{itemize}\setlength\itemsep{4pt}
\item If $m \in \calX_D$ then $\delta_a (m) = [\rho_1^+(m) - \rho_2^+(m)] [\rho_1^-(m) - \rho_2^-(m)]$.
\item If $m\in \calY_D$ then $\delta_a (m) = -[\rho_2^+(m) - 2 \rho_1^+(m)] [\rho_2^-(m)- 2 \rho_1^-(m)]$.
\item If $m \in \overline{\calZ}_D$ then $\delta_a (m) = \rho_s^\downarrow (m) \rho_p^\downarrow(m) + \rho_s^\uparrow(m) \rho_p^\uparrow(m) - 3 \rho_p^\downarrow(m) \rho_p^\uparrow(m)$. 
\item If $m\in\calF_D$ then $\delta_a(m)=0$, i.e. $m$ does not change.
\item If $m\in\mathcal{I}$ then $m'=\rho_R^+(m)'\rho_A^+(m)'+\rho_R^-(m)'\rho_A^-(m)'$.
\end{itemize}
\end{theorem}

\begin{proof}
The first four cases follow from the main theorem for mutations of type $A$, Theorem~\ref{thmA}, and the fact that projections in this case agree with those in type $A$,  see Lemma~\ref{projlemma}.  The last case follows from Lemma~\ref{lastcase}, since $\mathcal{I} \subset \calB$.  
\end{proof}

\begin{remark}
In order to mutate entries in $\mathcal{I}$ we need to apply the theorem above four times.  However, the mutation of projections of $m\in\mathcal{I}$ onto $\calR$ and $\calA$ fall into the first four cases of the theorem, so they can be computed directly from the given formulas.
\end{remark}

The previous theorem, together with Theorem \ref{mut-thm} and Remark \ref{properties}, allows us to mutate friezes of type $D$. Notice that $g(\frakF_{\T})$ has only one copy of $\calA$, while two copies appear in $\frakG_{\T}$. The reader can choose the easiest way to mutate entries in $\calA$, applying the theorem  to the entry in $\calA$ or the entry in $\hat{\calA}$.

\begin{figure}[h!]
\begin{center}
\scalebox{.8}{\hspace*{-0.6cm}
\def\svgwidth{4.6in}
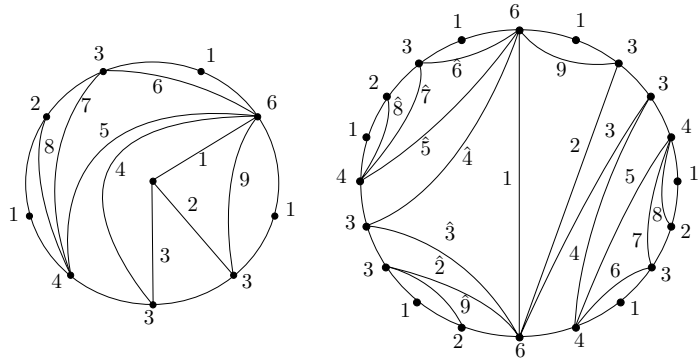}
\end{center}
\caption{Triangulation of ${\bf S}$ and the corresponding $\tilde{\bf S}_1$.}  
\label{last example triangulations}
\end{figure}

\subsection{An example}  In this section we provide a detailed example of how to compute entries in a frieze of type $D$ after mutation using Theorem~\ref{main-theorem}.  Consider the triangulation of the punctured disk in Figure~\ref{last example triangulations} and its corresponding frieze of type $D_9$ appearing on the left in Figure~\ref{frieze-ex}.

Consider the unique entry $m=104$ in $\frakG_{\T}$.  Let $a=5$, and we show how to compute $m'$ after mutation at 5.  Observe that $m\in \mathcal{I}$, so Theorem~\ref{main-theorem} implies that 
$$m'=\rho_R^+(m)'\rho_A^+(m)'+\rho_R^-(m)'\rho_A^-(m)'.$$
Note that 
$$\rho_R^+(m)= 7\hspace{1cm} \rho_A^+(m) =2 \hspace{1cm} \rho_R^-(m)= 5 \hspace{1cm} \rho_A^-(m)=18$$
are the circled entries in Figure~\ref{final-example}, each of these projections belongs to $\overline{\mathcal{Z}}_D$.  By Thoerem~\ref{main-theorem}, to compute these entries after mutation at $a$ we need to find their respective projections $\rho_s^\downarrow (\cdot ),  \rho_p^\downarrow(\cdot), \rho_s^\uparrow(\cdot), \rho_p^\uparrow(\cdot)$.  However, as shown below a lot of these projections coincide.

\begin{figure}[h!]
\begin{center}
\hspace*{-0.8cm}
\scalebox{0.72}{\def\svgwidth{9.7in}
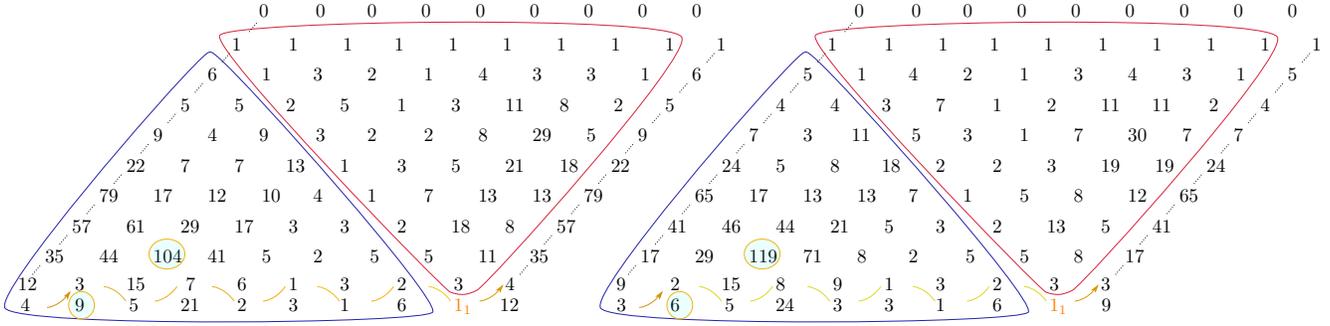}
\end{center}
\caption{Frieze of type $D_9$ coming from the triangulation in Figure~\ref{last example triangulations} and its mutation at $a=5$.}  
\label{frieze-ex}
\end{figure}

\begin{figure}[h!]
\begin{center}
\hspace*{-0.6cm}
\scalebox{0.85}{\def\svgwidth{7in}
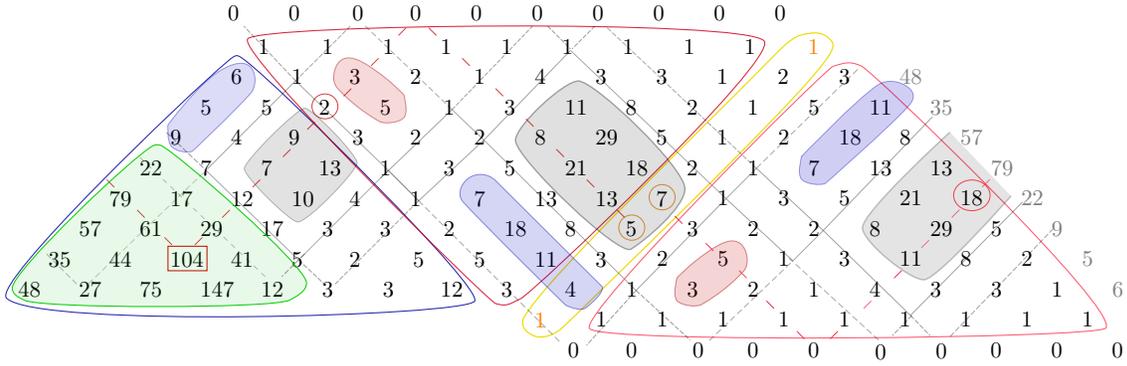}
\end{center}
\caption{Mutation of $m=104$, and the corresponding projections $\rho_R^+(m), \rho_A^+(m), \rho_R^-(m), \rho_A^-(m)$.}  
\label{final-example}
\end{figure}

$$\rho_p^\uparrow(5)=\rho_p^\downarrow(2)=\rho_A^+(m) =2$$        
$$\rho_p^\downarrow(5)=\rho_p^\uparrow(7)=2$$
$$\rho_s^\uparrow(5)=\rho_s^\downarrow(7)=3   $$
$$\rho_s^\downarrow(5)=\rho_s^\uparrow(2)=3$$
$$\rho_s^\uparrow(7)=\rho_s^\downarrow (18)=4$$
$$\rho_p^\downarrow(7)=\rho_p^\uparrow(18)=3$$

Also, 
$$\rho_p^\uparrow(2) = 1  \hspace{1cm} \rho_s^\downarrow(2)=1 \hspace{1cm} \rho_p^\downarrow(18)=5 \hspace{1cm} \rho_s^\uparrow(18)=8$$
Therefore, we obtain 
$$\rho_R^+(m)'=7- [\rho_s^\downarrow (7) \rho_p^\downarrow(7) + \rho_s^\uparrow(7) \rho_p^\uparrow(7) - 3 \rho_p^\downarrow(7) \rho_p^\uparrow(7)]=8.$$
Similar computations imply $\rho_A^+(m')=3, \rho_R^-(m)'=5, \rho_A^-(m)'=19$. We finally conclude that $m'=119$.

\smallskip

Now, consider the only entry $n=9$ in the double row of our initial frieze. By Remark \ref{properties}, this entry is related to $n_1 =27 \in \mathcal{I}$  and $n_2 =3 \in \calR$ in $\frakG_{\T}$. Then $n'$ is the quotient $n'_1 / n'_2$. For $n_1= 27$ the mutation is obtained as in the previous case
$$n'_1= \rho_R^+ (n_1)' \rho_A^+(n_1)' + \rho_R^- (n_1)' \rho_A^- (n_1)' = 12,$$   
and while we compute this, we obtain
$$\delta_5(3)= \rho_s^\downarrow (3) \rho_p^\downarrow(3) + \rho_s^\uparrow(3) \rho_p^\uparrow(3) - 3 \rho_p^\uparrow(3) \rho_p^\downarrow(3) = 1. $$
It follows that the mutation of $n_2$ is $n'_2 = 2$, thus  $n' = 12 / 2 = 6.$

\begin{remark}
We are also interested in studying the underlying representation theory of the construction that takes type $D$ to type $A$.   In the forthcoming work, we investigate the precise relationship between the two cluster-tilted algebras and their corresponding module categories.  
\end{remark}

\vskip .2cm

\noindent{\bf Acknowledgements:~}
A.G.E.  was supported by the Austrian Science Fund Project Number P30549, and she would also like to thank Department of Mathematics at UC Berkeley for inviting her for a two-week research visit, where this project was completed.
K.S. acknowledges partial support from the National Science Foundation Postdoctoral Fellowship MSPRF-1502881.
The authors also thank the anonymous referees for their comments and suggestions on the article.

\bibliographystyle{alpha}
\bibliography{friezes-5}

\end{document}